\documentclass[12pt]{amsart} 
\usepackage{amsmath, amssymb,graphics}
\usepackage[colorlinks=true]{hyperref}
\usepackage{fullpage}
 
\newcommand{\bbC}{\mathbb{C}}
\newcommand{\bbK}{\mathbb{K}}
\newcommand{\bbQ}{\mathbb{Q}}
\newcommand{\bbR}{\mathbb{R}}
\newcommand{\bbZ}{\mathbb{Z}}

\newcommand{\cC}{{\mathcal{C}}}
\newcommand{\cD}{{\mathcal{D}}}
\newcommand{\cN}{{\mathcal{N}}}
\newcommand{\cO}{{\mathcal{O}}}

\newcommand{\Chr}{\mathrm{Chr}}

\DeclareMathOperator{\core}{sf}
\DeclareMathOperator{\Gal}{Gal}

\DeclareMathOperator{\Imag}{Im}
\DeclareMathOperator{\Real}{Re}

\newtheorem{theorem}{Theorem}[section]

\newtheorem{lemma}[theorem]{Lemma}

\theoremstyle{remark}
\newtheorem*{note-nonum}{Note}
\newtheorem{remark}{Remark}

\numberwithin{equation}{section}

\begin{document}

\title[Improved Constants for the Hypergeometric Method]{Improved Constants for Effective Irrationality Measures from Hypergeometric Functions}

\author{Paul M. Voutier}
\address{London, UK \\Paul.Voutier@gmail.com}

\begin{abstract}
We simplify and improve the constant, $c$, that appears in
effective irrationality measures,
\[
\left| (a/b)^{m/n}-p/q \right|>c|q|^{-(\kappa+1)},
\]
obtained
from the hypergeometric method for $a/b$ near $1$. The dependence of $c$ on
$|a|$ in our result is best possible (as is the dependence on $n$ in many cases).
For some applications, the dependence of this constant on $|a|$ becomes important.
We also establish some new inequalities for hypergeometric
functions that are useful in other diophantine settings.
\end{abstract}

\subjclass[2010]{11J82, 11J68, 33C05}
\keywords{Diophantine Approximation, Effective Irrationality Measures, Hypergeometric Functions}

\maketitle

\section{Introduction}

Hypergeometric functions have played an important role in addressing diophantine
problems since the work of Thue. It was Siegel \cite{Si1} who first recognised
that the functions Thue used were hypergeometric functions. Siegel also
refined Thue's ideas and used hypergeometric functions himself. For example, he
used them to investigate the integer solutions of Thue equations involving binomial
forms (i.e., $ax^{n}-by^{n}=c$).
This work was developed further by Evertse \cite{Ev} and others, most notably by
Bennett \cite{Ben1}.

Baker \cite{Bak} was the first to show that hypergeometric functions can also be
used to obtain effective irrationality measures for rational powers of certain
rational numbers (although see Section~3.5 of \cite{FN} for how close Thue \cite{Thue2}
came to establishing some such results nearly 50 years earlier when he obtained
explicit upper bounds for the size of solutions of some Thue inequalities of the
form $\left| aq^{r}-bp^{r} \right| \leq k$). For example, he proved that
\[
\left| 2^{1/3} - \frac{p}{q} \right| > \frac{10^{-6}}{|q|^{2.955}},
\]
for all integers $p$ and $q$ with $q \neq 0$. Since then, his technique has
been improved, notably through Chudnovsky's analysis of denominators of the
coefficients of the associated hypergeometric functions \cite{Chud}. 

There is also a generalisation of the ordinary hypergeometric method developed
by Baker, known as Thue's
Fundamentaltheorem (from the title of his paper on it \cite{Thue1}). It can apply to cases not
covered by the former.

In previous work
\cite{Vout2, Vout3, Vout4}, we simplified the statement of Thue's
Fundamentaltheorem and investigated the conditions under which it yields
effective irrationality measures for algebraic numbers.
The focus in these papers was primarily on the irrationality exponent.
But for some problems, it can also be important to have good values for the constant
term too ($c$ in Theorem~\ref{thm:gen-hypg2} below), in particular a good dependence
on the quantity $a$ in Theorem~\ref{thm:gen-hypg2}. We consider that in this paper.

Furthermore, we obtain some lower bounds for the hypergeometric functions involved,
as well as lower bounds for their denominators (where appropriate). These have
played an important role in some forthcoming works of the author, so hopefully
they will be helpful for other diophantine problems and perhaps even in other areas too.

\section{Results}
\label{sect:results}

To present our results, we begin with some notation.
For relatively prime positive integers $m$ and $n$ with $0<m<n/2$ and $n \geq 3$,
and a non-negative integer $r$, we put
\[
X_{m,n,r}(z) = {} _{2}F_{1}(-r,-r-m/n;1-m/n;z)
\hspace*{3.0mm} \text{and} \hspace*{3.0mm} Y_{m,n,r}(z)=z^{r}X_{m,n,r}(1/z),
\]
where $_{2}F_{1}$ denotes the classical hypergeometric function. The condition
$m<n/2$, rather than $m<n$, poses no real restriction, and is necessary for the
proof of Lemma~\ref{lem:hypg-LB}.

Since $-r$ is a non-positive integer, $X_{m,n,r}(z), Y_{m,n,r}(z) \in \bbQ[z]$.
We let $D_{m,n,r}$ be a positive integer such that $D_{m,n,r} X_{m,n,r}(z) \in \bbZ[z]$.

For a non-negative integer $r$ and non-zero $d \in \bbZ$, we let $N_{d,m,n,r}$
be a positive integer such that $\left( D_{m,n,r}/ N_{d,m,n,r} \right)X_{m,n,r}\left( 1-\sqrt{d}\,z \right)
\in \bbZ \left[ \sqrt{\core(d)} \right] [z]$.
Here $\core(d)$ is the unique squarefree integer
such that $d/\core(d)$ is a square, with $\core(1)=1$.

This fixes a notational error in \cite{Vout4} (fixed in arXiv link provided),
although the proofs and the results there are correct and unaffected. This is
also a slight improvement on the definition of what
should be denoted as $N_{m,d,n,r}$ in \cite{Vout4}, affecting only the constant
in our results. In practice, when applying our results below in $\bbQ \left( \sqrt{t} \right)$,
we will take $d$ as a suitable square multiple of $t$. In this way, the sequence
of approximations we obtain in the course of the proof will be algebraic integers
in $\bbQ \left( \sqrt{t} \right)$, as required. This explains the choice of $d$
in Theorem~\ref{thm:gen-hypg2}.

We will use
$v_{p}(x)$ to denote the exponent of the largest power of a prime $p$ which divides into the rational number $x$. We put
\begin{equation}
\label{eq:ndn-defn}
\cN_{d,n} =\prod_{p|n} p^{\min(v_{p}(d)/2, v_{p}(n)+1/(p-1))},
\end{equation}
and choose real numbers $\cC_{n} \geq 1$ and $\cD_{n}>0$ such that
\begin{equation}
\label{eq:cndn-defn}
\max_{\stackrel{0<m<n/2}{\gcd(m,n)=1}} \left( \max \left( 1, \frac{\Gamma(1-m/n) \, r!}{\Gamma(r+1-m/n)},
\frac{n\Gamma(r+1+m/n)}{m\Gamma(m/n)r!} \right)
\frac{D_{m,n,r}}{N_{d,m,n,r}} \right)
< \cC_{n} \left( \frac{\cD_{n}}{\cN_{d,n}} \right)^{r}
\end{equation}
holds for all non-negative integers $r$, where $\Gamma(x)$ is the Gamma function.
This condition on $\cC_{n}$ and $\cD_{n}$ also corrects the one given in \cite{Vout4}.

In what follows, for $z$ not on the negative real line, when we take
a root of $z$, we mean the principal value of the root. I.e., writing $z=se^{i\varphi}$,
where $s$ is a non-negative real number and $-\pi<\varphi \leq \pi$ (with $\varphi=0$
when $s=0$), $z^{1/n}$ will signify $s^{1/n}e^{i\varphi /n}$ for a positive integer $n$,
where $s^{1/n}$ is the unique non-negative real $n$-th root of $s$.

\begin{theorem}
\label{thm:gen-hypg2}
Let $\bbK$ be an imaginary quadratic field with $m$ and $n$ as above.
Let $a$ and $b$ be algebraic integers in $\bbK$ with
either $0<b/a<1$ a rational number or $|b/a|=1$ with $0<|b/a-1|<1$. Let
$\cC_{n}$, $\cD_{n}$ and $\cN_{d,n}$ be as above with $d=(a-b)^{2}$. Put
\begin{eqnarray*}
E      & = & \frac{\cN_{d,n}}{\cD_{n}} \left\{ \min \left( \left| \sqrt{a}-\sqrt{b} \right|, \left| \sqrt{a}+\sqrt{b} \right| \right) \right\}^{-2}, \\
Q      & = & \frac{\cD_{n}}{\cN_{d,n}} \left\{ \max \left( \left| \sqrt{a}-\sqrt{b} \right|, \left| \sqrt{a}+\sqrt{b} \right| \right) \right\}^{2}, \\
\kappa & = & \frac{\log Q}{\log E} \hspace{3.0mm} \text{ and }\\
c      & = & 3|a|\cC_{n} \left( 20\cC_{n} \right)^{\kappa} \max \left( n, \cN_{d,n}^{\kappa} \right).
\end{eqnarray*}

If $E > 1$, then 
\[
\left| (a/b)^{m/n} - p/q \right| > \frac{1}{c |q|^{\kappa+1}} 
\]
for all algebraic integers $p$ and $q$ in $\bbK$ with $q \neq 0$. 
\end{theorem}

\begin{note-nonum}
The dependence on both $a$ and $n$ in $c$ is required. For example, if $n$ is an odd
integer, $a$ is a large positive integer and $b=a-1$, then the $0$-th convergent,
$p_{0}/q_{0}$, in the continued fraction expansion of $(a/b)^{1/n}$ is $1$
and the next partial quotient is $na-(n+3)/2$. So
$\left| (a/b)^{1/n} - p_{0}/q_{0} \right|$ is approximately $1/ \left( na \left| q_{0} \right|^{2} \right)$.
%
Similar results hold for other small-index convergents too.

An examination of the continued-fraction expansions of such numbers, suggests
that $O(|a|n)$ is the right size for $c$ for any value of $\kappa$ likely
to be obtained in the near-future. We obtain such a value here in cases that
commonly arise in applications like $a-b=1$, $(a-b,n)=1, \ldots$, when $\cN_{d,n}=1$,
so $c=3\cC_{n}\left( 20\cC_{n} \right)^{\kappa}|a|n$.
%
\end{note-nonum}

To make our theorem easier to use, we provide values for $\cC_{n}$ and $\cD_{n}$.
Since it is sometimes useful to have smaller values of $\cC_{n}$, we also give
$\cD_{2,n}$, the smallest calculated value of $\cD_{n}$ that allows us to take
$\cC_{n}=100$. For large $n$, $\cC_{1,n}<100$ for our choice of $\cD_{1,n}$.
For such $n$, we put $\cD_{2,n}=\cD_{1,n}$.

It is known (see Theorem~4.3 in \cite{Chud}) that $D_{m,n,r}$ has the
asymptotic behaviour
\[
\overline{\lim_{r \rightarrow \infty}} \frac{\log D_{m,n,r}}{r} \leq (\Chr)_{n}^{2},
\]
where
\[
(\Chr)_{n}^{2} = \frac{\pi}{\varphi(n)} \sum_{\substack{j=1\\(j,n)=1}}^{\lfloor n/2 \rfloor} \cot (\pi j/n).
\]

\begin{theorem}
\label{thm:dn-UB}
{\rm (a)} If $3 \leq n \leq 100$, then we can take
\[
\left( \cC_{n}, \cD_{n} \right) = \left( \cC_{1,n}, \cD_{1,n} \right)
\text{ or } \left( 100, \cD_{2,n} \right),
\]
where $\cC_{1,n}$, $\cD_{1,n}$ and $\cD_{2,n}$ are in Tables~$\ref{table:1}$--$\ref{table:3}$.

{\rm (b)} If $101 \leq n \leq 1009$ is prime and we consider only $m=1$ in \eqref{eq:cndn-defn},
then we can take
\[
\left( \cC_{n}, \cD_{n} \right) = \left( \cC_{1,n}, \cD_{1,n} \right)
\text{ or } \left( 100, \cD_{2,n} \right),
\]
where $\cC_{1,n}$, $\cD_{1,n}$ and $\cD_{2,n}$ are in Tables~$\ref{table:4}$--$\ref{table:7}$.

{\rm (c)} Otherwise, let $d_{1}=\gcd \left( d, n^{2} \right)$ and $d_{2}=\gcd \left( d/d_{1}, n^{2} \right)$.
If $d_{2}=1$, then
\[
\left( \cC_{n}, \cD_{n} \right) = \left( n, n \mu_{n} \right),
\]
where $\mu_{n}=\displaystyle\prod_{\stackrel{p|n}{p \, \, \mathrm{ prime}}} p^{1/(p-1)}$.
\end{theorem}

Since there are $\varphi(n)/2$ values of $m$ to consider for each value of $n$,
the work required to continue part~(a) for larger values of $n$ soon becomes prohibitive.
It is for this reason that we restrict to considering only $m=1$ for $101 \leq n \leq 1009$,
and also only consider $n$, prime, in this interval. Certainly for $n$ this large,
prime values of $n$ are the most important ones.

We stop at $n=1009$ only for the rather arbitrary reason that it is the smallest
prime greater than $1000$. In theory, using
Lemma~\ref{lem:bennett}, one could extend part~(b) to $n<1289$, as well as obtain
smaller values of $\cD_{n}$ in parts~(a) and (b).

\vspace*{1.0mm}

Before turning to the proof of these theorems, we also mention here some results obtained
in the course of the proof that may be of use to other researchers.\\
$\bullet$ Lemma~\ref{lem:approx}
improves on previous versions of this ``folklore lemma'' that is crucial for
obtaining effective irrationality measures from sequences of good approximations.
Here we use efficiently the $0$-th element in the sequence of good approximations
to replace the usual lower bound on $|q|$ with a (typically weak) condition on $\ell_{0}$
and $E$.\\
$\bullet$ Lemma~\ref{lem:7} provides a new lower bound for the hypergeometric
functions arising in analysis of the quality of our sequence of good
approximations. Moreover, it is best-possible where the hypergeometric method is
applicable.\\
$\bullet$ Lemma~\ref{lem:hypg-LB} provides a new lower bound for the hypergeometric
functions used in the construction of our sequence of good approximations.\\
$\bullet$ Lemma~\ref{lem:denom-LB1} provides a lower bound for the denominators of
the hypergeometric functions used in the construction of our sequence of good approximations.
It has the correct dependence on $n$.

\section{Preliminary Results}

The following lemma is used to obtain an effective approximation measure for
a complex number $\theta$ from a sequence of good approximations in an
imaginary quadratic field.

\begin{lemma}
\label{lem:approx}
Let $\theta \in \bbC$ and let $\bbK$ be an imaginary quadratic
field. Suppose that for all non-negative integers $r$, there are algebraic integers
$p_{r}$ and $q_{r}$ in $\bbK$ satisfying $p_{r}q_{r+1} \neq p_{r+1}q_{r}$ with
$\left| q_{r} \right|<k_{0}Q^{r}$ and $\left| q_{r}\theta-p_{r} \right| \leq \ell_{0}E^{-r}$,
for some real numbers $k_{0},\ell_{0} > 0$ and $E,Q > 1$ with $2\ell_{0}E \geq 1$.
Then for any algebraic integers $p$ and $q$ in $\bbK$ with $q \neq 0$ and
$p/q \neq p_{r}/q_{r}$ for all
non-negative integers $r$, we have
\[
\left| \theta - \frac{p}{q} \right| > \frac{1}{c |q|^{\kappa+1}},
\text{ where } c=2k_{0} \left( 2\ell_{0}E \right)^{\kappa}
\text{ and } \kappa = \frac{\log Q}{\log E}.
\]
\end{lemma}

\begin{note-nonum}
This is Lemma~6.1 in \cite{Vout2} with two changes. It has an improved value of
$c$ due to the additional condition that $p/q \neq p_{r}/q_{r}$ for all non-negative
integers $r$. We have also replaced the lower bound on $|q|$ with a lower bound
for $2\ell_{0}E$.

If we remove the restriction that $p/q \neq p_{r}/q_{r}$ for all non-negative
integers $r$, then the lemma still holds, but with $c$ above replaced by
$c=2k_{0}Q \left( 2\ell_{0}E \right)^{\kappa}$.
\end{note-nonum}

\begin{proof}
The proof is quite similar to that of Lemma~6.1 in \cite{Vout2}.

Let $p$, $q$ be algebraic integers in $\bbK$. If $|q| \geq 1/\left( 2\ell_{0} \right)$,
put
$\displaystyle r_{0} = \left\lfloor \frac{\log(2\ell_{0}|q|)}{\log E} \right\rfloor + 1$.
If $|q|<\left( 2\ell_{0} \right)$, then put $r_{0}=0$.
Note that in the first case, since $E>1$ and $2\ell_{0}|q| \geq 1$, we have
$r_{0} \geq 1$.

In the first case, it follows that $0 \leq \log \left( 2\ell_{0}|q| \right)/\log(E) < r_{0}$.
Hence, for all $r \geq r_{0}$,
\[
\ell_{0}E^{-r} < \ell_{0}E^{-(\log(2\ell_{0}|q|))/(\log E)} = 1/(2|q|) < 1,
\]
since $E>1$.

When $r_{0}=0$, then for all $r \geq r_{0}$,
\[
\ell_{0}E^{-r} \leq \ell_{0} <1/(2|q|) < 1,
\]
since $E>1$ and every non-zero algebraic integer in $\bbK$ has absolute value at least $1$.

In both cases, we have
\begin{equation}
\label{eq:approx1}
\ell_{0}E^{-r} < 1/(2|q|) < 1,
\end{equation}
for all $r \geq r_{0}$.

If we have $q_{r}=0$ for some $r \geq r_{0}$, then from \eqref{eq:approx1},
$\left| p_{r} \right| = \left| q_{r} \theta - p_{r} \right| < \ell_{0}E^{-r}<1$,
which implies that $p_{r}=0$ (again, using the fact that all non-zero algebraic
integers in these fields are of absolute value at least $1$). This contradicts
the supposition that $p_{r}q_{r+1} \neq p_{r+1}q_{r}$.
Therefore, $q_{r} \neq 0$ for all $r \geq r_{0}$. 

So, for any $r \geq r_{0}$ with $p/q \neq p_{r}/q_{r}$, we have 
\begin{equation}
\label{eq:approx2}
\left| \theta - \frac{p}{q} \right| 
\geq \left| \frac{p_{r}}{q_{r}} - \frac{p}{q} \right| 
     - \left| \theta - \frac{p_{r}}{q_{r}} \right| 
\geq \frac{1}{|qq_{r}|} - \frac{\ell_{0}}{E^{r}|q_{r}|}
> \frac{1}{2|qq_{r}|},
\end{equation}
again using \eqref{eq:approx1} and the fact that $p_{r}q-q_{r}p$ is a non-zero
algebraic integer and hence of absolute value at least $1$ in such fields.

If $|q| \geq 1 / \left( 2 \ell_{0} \right)$, then the choice of $r_{0}$ yields
\begin{equation}
\label{eq:approx3}
Q^{r_{0}} \leq \exp \left( \frac{\log (2\ell_{0}|q|) + \log (E)}{\log(E)} \log (Q) \right)
= \left( 2E\ell_{0}|q| \right)^{\kappa}.
\end{equation}

If $|q| < 1 / \left( 2 \ell_{0} \right)$, so that $r_{0}=0$, then the same upper
bound holds for $Q^{r_{0}}=1$ by our assumption that $2E\ell_{0} \geq 1$ and hence
that $2E\ell_{0}|q| \geq 1$, since $q \neq 0$ implies that $|q| \geq 1$.

Combining \eqref{eq:approx2} and \eqref{eq:approx3} with our upper bound in the
lemma for $\left| q_{r_{0}} \right|$, we have
\[
\left| \theta - \frac{p}{q} \right| 
> \frac{1}{2|q q_{r_{0}}|} 
\geq \frac{1}{2|q|k_{0}Q^{r_{0}}} 
\geq \frac{1}{2k_{0}(2E\ell_{0})^{\kappa}|q|^{\kappa+1}},
\]
when $p/q \neq p_{r_{0}}/q_{r_{0}}$.
\end{proof}

For any non-negative integer, $r$, let
\begin{equation}
\label{eq:Rmnr-defn}
R_{m,n,r}(z) = \frac{(m/n) \cdots (r+m/n)}{(r+1) \cdots (2r+1)} 
	       {} _{2}F_{1} \left( r+1-m/n, r+1; 2r+2; 1-z \right).
\end{equation}

The next lemma contains the relationship that allows the hypergeometric method to provide
good sequences of rational approximations.

\begin{lemma}
\label{lem:relation}
Let $m,n$ and $r$ be non-negative integers with $0<m<n$ and $\gcd(m,n)=1$.
If $z$ is any complex number with $|z| \leq 1$ and $|z-1|<1$, then
\begin{equation}
\label{eq:approx}
Y_{m,n,r}(z) - (1/z)^{m/n}X_{m,n,r}(z) = z^{-m/n}(z-1)^{2r+1} R_{m,n,r}(z). 
\end{equation}
\end{lemma}

\begin{proof}
This is a slight variation on equation (4.2) in \cite{Chud} with $\nu=m/n$.
We multiply that equation by $(1/z)^{m/n}$ to obtain \eqref{eq:approx}.

The reason for this change is that we have an easy upper bound for
$\left| X_{m,n,r}(z) \right|$ when $0<b/a<1$ is a real number, so we will
use $X_{m,n,r}(z)$ to define our $q_{r}$ in Lemma~\ref{lem:approx}.
\end{proof}

\begin{lemma}
\label{lem:7}
Let $a$ and $b$ be positive real numbers with $b<2a$.
If $|z|=1$ and $|z-1|<1$, then we have
\begin{equation}
\label{eq:LB}
\left| {} _{2}F_{1} \left( a, b; 2a; 1-z \right) \right| \geq 1,
\end{equation}
with the minimum value occurring at $z=1$.
\end{lemma}

\begin{remark}
In fact, \eqref{eq:LB} appears to hold much more generally. Numerical experiments
suggest the following is true.
For all $z \in \bbC$ with
$|z| \leq 1$, $|1-z|<1$ and all $a,b,c \in \bbR$ satisfying $a,b>0$ and $\max (a,b)<c$,
we have $\left| {}_{2} F_{1} \left( a,b; c; 1-z \right) \right| \geq 1$.
\end{remark}

\begin{proof}
We proceed more generally initially.

Using Pochhammer's integral (see equation~(1.6.6) of \cite{Sl}), along with the
transformation $t=1/s$, we can write
\begin{align*}
{}_{2} F_{1} \left( a,b; c; z \right)
&=\frac{\Gamma(c)}{\Gamma(b)\Gamma(c-b)} \int_{1}^{\infty} (s-1)^{c-b-1}s^{a-c}(s-z)^{-a}ds \\
&=\frac{\Gamma(c)}{\Gamma(b)\Gamma(c-b)} \int_{0}^{\infty} s^{c-b-1}(s+1)^{a-c}(s+1-z)^{-a}ds,
\end{align*}
provided that $|z|<1$, $\Real(c-b)>0$ and $\Real(b)>0$.

Therefore, we can write
\[
{}_{2} F_{1} \left( a,b; c; 1-z \right)
=\frac{\Gamma(c)}{\Gamma(b)\Gamma(c-b)} \int_{0}^{\infty} s^{c-b}(s+1)^{a-c}(s+z)^{-a}ds/s
\]
for $|1-z|<1$ and our problem becomes one of showing that the absolute value of the function
\begin{equation}
\label{eq:hypg-integral1}
\int_{0}^{\infty} t^{\alpha} (t+1)^{-\beta}(t+z)^{-\gamma}\frac{dt}{t}
\end{equation}
with $\alpha, \beta, \gamma>0$ and $\beta+\gamma>\alpha$ attains its minimum for
$|z|=1$ with $\Real(z) \geq 0$ at $z=1$.

Note that here we have $\alpha=c-b$, $\beta=c-a$ and $\gamma=a$.

We can change the integration path to any path that avoids the singularities of
the integrand in \eqref{eq:hypg-integral1}, i.e., any path that stays in the open
angle bounded by the rays $\{-\tau z:\tau>0\}$ and $\{-\tau:\tau>0\}$ containing
the positive semi-axis. So we will change it to the ray $\{\tau\sqrt{z}: \tau>0\}$
(here, as elsewhere, we use the principal value of the square root).
Thus the integral in \eqref{eq:hypg-integral1} becomes
\[
\int_{0}^{\infty} \left( \sqrt{z} t \right)^{\alpha}
\left( \sqrt{z}t+1 \right)^{-\beta} \left( \sqrt{z}t+z \right)^{-\gamma}
\frac{dt}{t}
=
z^{(\alpha-\beta-\gamma)/2}
\int_{0}^{\infty} t^{\alpha}
\left( t+1/\sqrt{z} \right)^{-\beta} \left( t+z/\sqrt{z} \right)^{-\gamma}
\frac{dt}{t}.
\]

Putting $w=1/\sqrt{z}$ and recalling that $|z|=1$, we have
\[
\left| z^{(\alpha-\beta-\gamma)/2}
\int_{0}^{\infty} t^{\alpha}
\left( t+1/\sqrt{z} \right)^{-\beta} \left( t+z/\sqrt{z} \right)^{-\gamma}
\frac{dt}{t} \right|
=
\left| \int_{0}^{\infty} t^{\alpha}
\left( t+w \right)^{-\beta} \left( t+wz \right)^{-\gamma}
\frac{dt}{t} \right|,
\]
so the problem is reduced to establishing the following:\\
for $w',z' \in \bbC$ with $\Real(w'), \Real(z')>0$ and $w'z' \in \bbR_{+}$, show
\begin{equation}
\label{eq:hypg-integral2}
\left| \int_{0}^{\infty} t^{\alpha}(t+w')^{-\beta}(t+z')^{-\gamma}\frac {dt}t \right|
\geq \left| \int_{0}^{\infty} t^{\alpha}(t+|w'|)^{-\beta}(t+|z'|)^{-\gamma}\frac {dt}t\right|.
\end{equation}

We now establish \eqref{eq:hypg-integral2} in the case of interest to us here.

Since $c=2a$, we have $\beta=\gamma$. By the definition above of $w$, we take
$w'=w=1/\sqrt{z}$ and $z'=zw=z/\sqrt{z}=\sqrt{z}$ in \eqref{eq:hypg-integral2}.
From the hypotheses that $|z|=1$ and $|1-z|<1$, it follows that $\Real(w'), \Real(z')>0$
and $w'z'=1 \in \bbR_{+}$. The integrand on the left-hand side of
\eqref{eq:hypg-integral2} is positive, since
$\left( t+w' \right)^{-\beta} \left( t+z' \right)^{-\gamma}
=\left( t^{2}+ \left( w'+z' \right) t + w'z' \right)^{-\gamma}
=\left( t^{2}+ 2\Real(w') t + 1 \right)^{-\gamma}$ and $0<2\Real(w')$.
That integrand is also greater
than the one on the right-hand side, since $2\Real(w') \leq |w'|+|z'|$. Hence
\eqref{eq:hypg-integral2} holds in this case.

Since the right-hand side of \eqref{eq:hypg-integral2} here is
\[
\left| \int_{0}^{\infty} t^{\alpha}(t+1)^{-\beta}(t+1)^{-\gamma}\frac {dt}t \right|,
\]
from Pochhammer's integral above, we see that the minimum value of 
$\left| {} _{2}F_{1} \left( a, b; 2a; 1-z \right) \right|$ occurs at $z=1$, where
it is equal to $1$.
\end{proof}

\begin{lemma}
\label{lem:hypg-UB}
Let $m,n$ and $r$ be non-negative integers with $0<m<n$ and $\gcd(m,n)=1$.

\noindent
{\rm (a)} If $|z|=1$ and $|z-1|<1$, then
\[
\left| X_{m,n,r}(z) \right|
< 1.072\frac{r!\Gamma(1-m/n)}{\Gamma(r+1-m/n)} \left| 1 + \sqrt{z} \right|^{2r}.
\]

\noindent
{\rm (b)} If $z \in \bbR$ with $0 \leq z \leq 1$, then
\[
\left| X_{m,n,r}(z) \right| < \left| 1 + \sqrt{z} \right|^{2r}.
\]
\end{lemma}

\begin{proof}
(a) This is a slight refinement of Lemma~7.3(a) of \cite{Vout2}.
In the proof of that lemma, we showed that in our notation here
\[
\left| X_{m,n,r}(z) \right|
\leq \frac{4}{\left| 1 + \sqrt{z} \right|^{2}} \frac{\Gamma(1-m/n) \, r!}{\Gamma(r+1-m/n)}
\left| 1+\sqrt{z} \right|^{2r}.
\]

Since $z$ is on the unit circle, we can write $1+\sqrt{z}=1+z_{1} \pm \sqrt{1-z_{1}^{2}}i$,
where $0 \leq z_{1} \leq 1$. Here we have $|\theta|<\pi/3$ in order that $|z-1|<1$
holds. Hence $z_{1}=\cos(\theta/2)>\cos(\pi/6)$, and so
\[
\frac{4}{\left| 1 + \sqrt{z} \right|^{2}}<1.072.
\]

(b) This is Lemma~5.2 of \cite{Vout1}, noticing that $Y_{m,n,r}(z)$ there is our
$X_{m,n,r}(z)$.
\end{proof}

In order to obtain the simplified constant in our effective irrationality measure,
we will also need lower bounds for the hypergeometric functions and for their
denominators. We now establish these results.

\begin{lemma}
\label{lem:hypg-LB}
Let $m,n$ and $r$ be non-negative integers with $0<m<n/2$ and $\gcd(m,n)=1$.
For $z \in \bbC$ with $\Real(z) \geq 0$, we have
\[
\left( 1 +\Real(z) \right)^{r} \leq \left| X_{m,n,r}(z) \right|,
\]
\[
\left( 1 +\Real(z) \right)^{r} \leq \left| Y_{m,n,r}(z) \right|.
\]
\end{lemma}

\begin{proof}
We start by showing that all the zeroes of $X_{m,n,r}(z)$ are negative real
numbers. Equation~(4.21.2) in \cite[p.~62]{Sz} defines the Jacobi polynomials,
$P_{n}^{(\alpha,\beta)}$ by
\[
P_{n}^{(\alpha,\beta)}(1-2z)= \binom{n+\alpha}{n}
{} _{2}F_{1} \left( -n, n+\alpha+\beta+1; \alpha+1; z \right).
\]

Here Szeg\H{o}'s $n$, $\alpha$ and $\beta$ correspond to our $r$, $-m/n$ and
$-2r-1$, respectively.

The zeroes of $X_{m,n,r}(z)$ will all be negative if the zeroes of
$P_{n}^{(\alpha,\beta)}(z)$ are all real and larger than $1$ for this choice of
$n$, $\alpha$ and $\beta$. The number of such zeroes is the quantity $N_{3}$ in
Theorem~6.72 in \cite[p.145]{Sz}. In the notation of this theorem,
$Z=\left[ r+1/2+m/n \right]=r$ -- this is one of the reasons why we need the
condition $m<n/2$. Furthermore, we see that Szeg\H{o}'s
\[
\binom{2n+\alpha+\beta}{n} \binom{n+\alpha}{n}
\]
is equal to
\[
\binom{-1-m/n}{n} \binom{r-m/n}{r}
\]
for our choice of $n$, $\alpha$ and $\beta$.
This quantity is negative if $r$ is odd and positive if $r$ is
even. Thus, by equation~(6.72.8) of \cite[p.146]{Sz}, $N_{3}=2 \left[ (r+1)/2 \right]=r$,
if $r$ is even and $N_{3}=2 \left[ r/2 \right]+1=r$ if $n$ is odd.

Since all the zeroes of $X_{m,n,r}$ are real, $z$ is always at least as
far from each of these zeroes as $\Real(z)$ is. Therefore,
$\left| X_{m,n,r}(z) \right| \geq \left| X_{m,n,r}(\Real(z)) \right|$. From
Lemma~5.2 of \cite{Vout1}, we have $\left| X_{m,n,r}(\Real(z)) \right| \geq (1+\Real(z))^{r}$,
as stated (note that $Y_{m,n,r}(z)$ in \cite{Vout1} is the same as $X_{m,n,r}(z)$
here).

Since $Y_{m,n,r}(z)=z^{r}X_{m,n,r} \left( z^{-1} \right)$, all its zeroes are
also negative real numbers, so we have
$\left| Y_{m,n,r}(z) \right| \geq \left| Y_{m,n,r}(\Real(z)) \right|$ too.
Since the coefficient of $z^{k}$ in $(1+z)^{r}$ equals the coefficient of $z^{r-k}$,
the argument in Lemma~5.2 of \cite{Vout1} showing that
$\left| X_{m,n,r}(\Real(z)) \right| \geq (1+\Real(z))^{r}$ also shows that
$\left| Y_{m,n,r}(\Real(z)) \right| \geq (1+\Real(z))^{r}$.
\end{proof}

Important for our work will be the following result of \cite{Ben2}.

\begin{lemma}
\label{lem:bennett}
Suppose $m$ and $n$ are relatively prime rational integers with $3 \leq n \leq 10^{4}$
and $0<m<n$. Recall that $\displaystyle\theta(x;n,m)=\sum_{\stackrel{p \leq x}{p \equiv m \bmod n}} \log(p)$,
where the sum is over all such primes $p$.

We have
\begin{equation}
\label{eq:theta-UB}
\left| \theta(x;n,m) - \frac{x}{\varphi(n)} \right| < \frac{x}{840\log(x)} <
\left\{
\begin{array}{ll}
4.31 \cdot 10^{-5}x & \text{for $x \geq 10^{12}$},\\
3.98 \cdot 10^{-5}x & \text{for $x \geq 10^{13}$}.
\end{array}
\right.
\end{equation}

Furthermore,
\begin{equation}
\label{eq:theta-UBsqrt}
| \theta(x;n,m) - x/\varphi(n)| < 1.818 \sqrt{x},
\end{equation}
for $x \leq 10^{12}$ and each $3 \leq n \leq 10^{4}$ (for $x \leq 10^{13}$ when
$3 \leq n \leq 100$).
\end{lemma}

\begin{proof}
Equation~\eqref{eq:theta-UB} follows from Theorem~1.2 of \cite{Ben2}.

Equation~\eqref{eq:theta-UBsqrt} follows from Theorem~1.9 and equation~(A.2) of
\cite{Ben2}.
\end{proof}

\begin{lemma}
\label{lem:denom-LB1}
Let $m$, $n$ and $r$ be non-negative integers with $0<m<n/2$, $\gcd(m,n)=1$ and $n \geq 3$.

\noindent
{\rm (a)} We have
\begin{equation}
\label{eq:denom-LB1a}
D_{m,n,r} > (n/4)^{r} \cdot \prod_{\stackrel{p|n}{\text{$p$, prime}}} p^{v_{p}((2r)!)-v_{p}(r!)}
\geq \left( n\mu_{n}/4 \right)^{r}(2r+1)^{-\omega(n')/2},
\end{equation}
where $\mu_{n}$ is as in Theorem~$\ref{thm:dn-UB}$, $n'$ is the largest odd factor
of $n$ and $\omega(n')$ is the number of distinct prime factors of $n'$.

\noindent
{\rm (b)} We have
\begin{equation}
\label{eq:denom-LB1b}
D_{m,n,r} > \left\{
\begin{array}{ll}
0.08  \cdot 2.1^{r}  & \text{if $n=3$},\\
0.02  \cdot 3.77^{r} & \text{if $n=4$},\\
0.3   \cdot 2.54^{r} & \text{if $n=5$},\\
0.3   \cdot 10.9^{r} & \text{if $n=6$},\\
0.7   \cdot 2.63^{r} & \text{if $n=7$},\\
0.2   \cdot 5.53^{r} & \text{if $n=8$}.
\end{array}
\right.
\end{equation}
\end{lemma}

\begin{note-nonum}
The lower bound in \eqref{eq:denom-LB1a}, while smaller than the asymptotics of $D_{m,n,r}$,
does show the right dependence on $n$. In Remark~7.7 of \cite{Vout2}, we stated
that $n\mu_{n}$ is approximately $\left( \pi/e^{\gamma} \right) (\Chr)_{n}^{2}$,
so $n\mu_{n}/4$ is approximately $0.44 (\Chr)_{n}^{2}$.
\end{note-nonum}

\begin{proof}
(a) Our proof uses the fact that if $f(z) \in \bbQ[z]$, then the least common multiple
of the denominators of its coefficients must be at least the reciprocal of the
absolute value of $f(v)$ for an integer $v$. Since ${}_{2}F_{1}(a,b;c;1)$ has a
nice value, we consider $v=1$ here.

Using the Chu-Vandermonde identity (see equation~(15.5.24) 
of \cite{DLMF}) with $b=-r-m/n$, $c=1-m/n$ and $n$ there equal to our $r$, we have
\[
X_{m,n,r}(1)
=\frac{(c-b)_{r}}{(c)_{r}}
=\frac{(r+1)\cdots(2r)}{(1-m/n)\cdots (r-m/n)}
=n^{r}\frac{(r+1)\cdots(2r)}{(n-m)\cdots (rn-m)},
\]
where $(a)_{r}=a \cdots (a+r-1)$ is Pochhammer's symbol. Since $n$ is relatively
prime to the denominators of the coefficients of $X_{m,n,r}(z)$ (all the terms
in the denominators are of the form $in-m$ and $\gcd(m,n)=1$), we need only
consider $(r+1)\cdots(2r)/\left[ (n-m)\cdots (rn-m) \right]$.

Now
\[
\frac{(n-m)\cdots (nr-m)}{(r+1)\cdots (2r)}
>\frac{(n/2)\cdots (nr-n/2)}{(r+1)\cdots (2r)}
=\frac{n^{r}(2r-1)!}{2^{2r-1}(r-1)!}\frac{r!}{(2r)!}
=(n/4)^{r},
\]
since $m<n/2$.

But we can remove more powers of prime divisors of $n$ from $(r+1)\cdots (2r)$ too:
\[
\prod_{\stackrel{p|n}{\text{$p$, prime}}} p^{v_{p}((2r)!)-v_{p}(r!)} | \left( (r+1)\cdots(2r) \right).
\]

Observe that $(2r)!/r!=2^{r} \cdot 1 \cdot 3 \cdots (2r-1)$, so
$2^{v_{2}((2r)!)-v_{2}(r!)}=2^{r}$. Letting $s_{p}(r)$ be the sum of the digits
in the base $p$ expansion of $r$, we find that $v_{p}((2r)!)-v_{p}((r!))=(2r-s_{p}(2r))/(p-1)
-(r-s_{p}(r))/(p-1)$ (see Exercise~14 on page~7 of \cite{Kob}). So
\[
v_{p}((2r)!)-v_{p}((r!))=r/(p-1)+ \left( s_{p}(r)-s_{p}(2r) \right)/(p-1),
\]
for primes $p>2$. The maximum of $p^{(s_{p}(2r)-s_{p}(r))/(p-1)}$ occurs when all the
base $p$ digits of $2r$ are equal to $p-1$, in which case it is $\sqrt{2r+1}$.
Part~(a) of the lemma follows.

\vspace*{3.0mm}

(b) Taking $A=0$ and $\ell=1$ in Lemma~3.3(b) of \cite{Vout1}, we know that if
$p>(nr+m)^{1/2}$ is a prime such that $p \equiv -m \bmod n$ and
\[
\frac{nr+m+n}{n-1} \leq p \leq nr-m,
\]
then $p | D_{m,n,r}$. Furthermore, if $r \geq n$, then $(nr+m+n)/(n-1)>(nr+m)^{1/2}$.
This is why we calculate the values for $r \leq n$.

Thus
\begin{align}
\label{eq:dmnr-LB}
\log D_{m,n,r}
& \geq \sum_{\stackrel{(nr+m+n)/(n-1) \leq p \leq nr-m}{p \equiv -m \bmod n}} \log(p) \\
&= \theta \left( nr-m; n, -m \right) - \theta \left( (nr+m+n)/(n-1); n, -m \right), \nonumber
\end{align}
where $\theta(x;n,-m)$ is the sum of the logarithms of all primes $p \leq x$
with $p \equiv -m \bmod n$ and $\varphi(n)$ is Euler's phi function.

From Corollary~1.7 of \cite{Ben2}, we have
\[
\left| \theta(x; n, -m) - \frac{x}{\varphi(n)} \right|<0.00174x,
\]
for $x \geq 10^{6}$ and the pairs $(m,n)$ being considered here.

We apply this inequality to \eqref{eq:dmnr-LB} to obtain the bounds in the lemma
for $r \geq 10^{6}$. In fact, we obtain inequalities where the constants in front
of the exponential terms are slightly larger.

Using a program written in Java, we computed the denominators for the remaining
polynomials. This took just under 500 seconds on a Windows laptop with an Intel
i7-9750H 2.60GHz CPU. It is from this calculation that the constants in front
of the exponential terms arise. E.g., for $n=3$, we had to replace the constant
$0.1$ with $0.08$, which is required for $m=1$ and $r=13$. Part~(b) follows.
\end{proof}

The following lemma is much weaker than Lemma~\ref{lem:denom-LB1} permits (roughly
$\left( n\mu_{n}/8 \right)^{2r}$), but it suffices for our needs here.

\begin{lemma}
\label{lem:denom-LB2}
Let $m,n$ and $r$ be non-negative integers with $0<m<n/2$, $n \geq 3$ and $\gcd(m,n)=1$.
We have
\begin{equation}
\label{eq:denom-LB2}
\frac{m}{60n} < D_{m,n,r}^{2}\frac{(m/n) \cdots (r+m/n)}{(r+1) \cdots (2r+1)}
\end{equation}
\end{lemma}

\begin{note-nonum}
We do not consider $r=0$ here, as the right-hand side is $m/n$ in this case and
hence dependent on $n$, whereas we want an absolute constant on the left-hand side
in our result.
\end{note-nonum}

\begin{proof}
We first consider $r=0$. Here $X_{m,n,r}(z)=1$, so $D_{m,n,r}=1$ and the right-hand
side of \eqref{eq:denom-LB2} is $m/n$. So the lemma holds in this case.

Now consider $r=1$. Here $X_{m,n,r}(z)=(n+m)z/(n-m)+1$, so $D_{m,n,r}=(n-m)/2$
if $m$ and $n$ are both odd and $n-m$ otherwise, since $m$ and $n$ are relatively
prime. So right-hand side of \eqref{eq:denom-LB2} is at least
$m(n+m)(n-m)^{2}/\left( 24n^{2} \right)$. Taking the derivative of this quantity
with respect to $m$, we obtain
\[
\frac{(n-m) \left( n^{2}-nm-4m^{2} \right)}{24n^{2}}.
\]

Its numerator is zero when $m=n$ and $m= \left( -1 \pm \sqrt{17} \right)n/8$.
Only $m= \left( -1 + \sqrt{17} \right)n/8$ satisfies $0<m<n/2$, so we consider
this value of $m$ (where $m(n+m)(n-m)^{2}/\left( 24n^{2} \right)
=0.0084\ldots n^{2}>1/14$ for $n \geq 3$), along with $m=1$
(where $m(n+m)(n-m)^{2}/\left( 24n^{2} \right)
=(n+1)(n-1)/\left( 24n^{2} \right) \geq 2/27$ for $n \geq 3$)
and $m=(n-1)/2$ (where $m(n+m)(n-m)^{2}/\left( 24n^{2} \right)
=(n-1)(3n-1)(n+1)^{2}/\left( 384n^{2} \right) \geq 2/27$ for $n \geq 3$). So the lemma
holds for $r=1$.

We need a lower bound for
$(m/n) \cdots (r+m/n)/ \left( (r+1) \cdots (2r+1) \right)$.
We can write
\[
\frac{(m/n) \cdots (r+m/n)}{(r+1) \cdots (2r+1)}
= \frac{\Gamma(r+1+m/n)\Gamma(r+1)}{\Gamma(m/n)\Gamma(2r+2)}.
\]

We will use
\[
1<(2\pi)^{-1/2}x^{(1/2)-x}e^{x}\Gamma\left(x\right)<e^{1/(12x)},
\]
(see inequality~(5.6.1) in \cite{DLMF}).

Applying these inequalities to each of these four gamma function values,
we obtain
\[
\frac{(m/n) \cdots (r+m/n)}{(r+1) \cdots (2r+1)}
> \left( \frac{r+1+m/n}{r+1} \right)^{r+1} (r+1+m/n)^{m/n-1/2}
\frac{2^{-2r-3/2}}{(m/n)^{m/n-1/2}e^{(n/m+1/(2r+2))/12}}.
\]

Simplifying this, we find that
\begin{equation}
\label{eq:gamma1}
\frac{(m/n) \cdots (r+m/n)}{(r+1) \cdots (2r+1)}
> 4^{-r}\sqrt{m/(8n)} r^{m/n-1/2}>4^{-r}\sqrt{m/(8nr)}.
\end{equation}

We combine this with the lower bound, $D_{m,n,r}>(n/4)^{r}$, which follows from
\eqref{eq:denom-LB1a} in Lemma~\ref{lem:denom-LB1}(a):
\[
(n/8)^{2r}\sqrt{m/(8rn)}
< D_{m,n,r}^{2}\frac{(m/n) \cdots (r+m/n)}{(r+1) \cdots (2r+1)}.
%
\]

For $n \geq 9$, the left-hand side is greater than $m/(60n)$ for $r \geq 2$.

For $n \leq 8$, we use Lemma~\ref{lem:denom-LB1}(b). Writing the lower bounds
there as $d_{1}\cdot d_{2}^{r}<D_{m,n,r}$ and using \eqref{eq:gamma1}, the right-hand
side of \eqref{eq:denom-LB2} is greater than
\[
\sqrt{m/(8rn)}d_{1} \left( d_{2}/2 \right)^{2r}.
\]

For $3 \leq n \leq 8$, we can easily calculate that this quantity is greater
than $m/(60n)$ for $n=3$ with $r \geq 20$; $n=4$ with $r \geq 4$; and $n=5,6,7,8$
with $r \geq 2$. and hence the lemma holds for such $n$ and $r$.

Computing the quantity on the right-hand side of \eqref{eq:denom-LB2} directly
for $n=3$ and $n=4$ and the remaining values of $r$ completes the proof of the
lemma.

The lower bound of $1/(60n)$ is nearly attained for $n=3$, $m=1$ and $r=13$, where
the value of the right-hand side is $0.00565\ldots$.
\end{proof}

\section{The approximations and their bounds}

We start by defining our sequence of approximations to $(a/b)^{m/n}$, along
with some estimates we will require.

Let $r$ be a non-negative integer, $a$ and $b$ be algebraic integers in an imaginary
quadratic field, $\bbK$, with either $0<b/a<1$ a rational number or $|b/a|=1$ with
$|b/a-1|<1$.
Put $d=(a-b)^{2}$. Motivated by Lemma~\ref{lem:relation}, we define
\begin{equation}
\label{eq:thm2-pr-qr}
q_{r} = \frac{a^{r} D_{m,n,r}}{N_{d,m,n,r}} X_{m,n,r}(b/a) 
\hspace{3.0mm} \text{ and } \hspace{3.0mm} 
p_{r} = \frac{a^{r} D_{m,n,r}}{N_{d,m,n,r}} Y_{m,n,r}(b/a)
=\frac{b^{r} D_{m,n,r}}{N_{d,m,n,r}} X_{m,n,r}(a/b).
\end{equation}

\begin{lemma}
\label{lem:thm2-pq-ub}
Let $r$ be a non-negative integer. Let $m$ and $n$ be relatively prime positive
integers with $0<m<n/2$. Then $p_{r}$ and $q_{r}$ are algebraic integers with 
$p_{r}q_{r+1} \neq p_{r+1}q_{r}$ and 
\begin{equation}
\label{eq:thm2-q-est}
\frac{D_{m,n,r}}{N_{d,m,n,r}} \left( |a| \left| 1+\Real(b/a) \right| \right)^{r} \leq q_{r} 
< 1.072\cC_{n} \left( \frac{\cD_{n}}{\cN_{d,n}} \right)^{r} \left| a^{1/2} + b^{1/2} \right|^{2r}.
\end{equation}
\end{lemma}

\begin{proof}
The assertion that $p_{r}$ and $q_{r}$ are algebraic integers is just a
combination of our definitions of $p_{r}$, $q_{r}$, $D_{m,n,r}$ and $N_{d,m,n,r}$.

That $p_{r}q_{r+1} \neq p_{r+1}q_{r}$ is equation~(16) in Lemma~4 of \cite{Bak}.

We now prove the upper bound for $q_{r}$.

If $b/a$ is a rational number with $0<b<a$ relatively prime integers, then from
Lemma~5.2 of \cite{Vout1} (recalling again that $Y_{m,n,r}(z)$ there is $X_{m,n,r}(z)$
here),
\[
a^{r} X_{m,n,r}(b/a) \leq { \left( a^{1/2} + b^{1/2} \right) }^{2r}.
\]

If $|b/a|=1$, then from Lemma~\ref{lem:hypg-UB},
\[
\left| \frac{a^{r} D_{m,n,r}}{N_{d,m,n,r}} X_{m,n,r}(b/a) \right|
< 1.072 \frac{D_{m,n,r}}{N_{d,m,n,r}} \frac{r!\Gamma(1-m/n)}{\Gamma(r+1-m/n)} \left| \sqrt{a} + \sqrt{b} \right|^{2r}.
\]

The upper bound for $q_{r}$ now follows from this and \eqref{eq:cndn-defn}.

The lower bound for $q_{r}$ is an immediate consequence of the 
lower bound for $X_{m,n,r}(z)$ in Lemma~\ref{lem:hypg-LB}.
\end{proof}

We next determine how close the resulting approximations are to $(a/b)^{m/n}$.

\begin{lemma}
\label{lem:thm2-8}
Let $m$ and $n$ be relatively prime positive integers with $m<n/2$ and $n \geq 3$.
Let $r$ be a non-negative integer and let $a$ and $b$ be algebraic integers in
an imaginary quadratic field, $\bbK$, with
either $0<b/a<1$ a rational number or $|b/a|=1$ with $|b/a-1|<1$. Then
\begin{equation}
\label{eq:thm2-remest}
\left| \frac{a-b}{60naq_{r}} \right|
< \left| q_{r} (a/b)^{m/n} - p_{r} \right| 
< 1.22 \left| \frac{a-b}{b} \right| \cC_{n} \left( \frac{\cD_{n}}{\cN_{d,n}} \right)^{r} \left| \sqrt{a}-\sqrt{b} \right|^{2r}. 
\end{equation}  
\end{lemma}

\begin{proof}
Using our definitions of $p_{r}$ and $q_{r}$ in \eqref{eq:thm2-pr-qr}, and of
$R_{m,n,r}(z)$ in \eqref{eq:Rmnr-defn}, along with the relation in \eqref{eq:approx}
in Lemma~\ref{lem:relation} with $z=b/a$, we find that
\begin{eqnarray*}
p_{r}-q_{r} (a/b)^{m/n}
& = & \left( \frac{a}{b} \right)^{m/n} \frac{a^{r} D_{m,n,r}}{N_{d,m,n,r}}
      { \left( \frac{b-a}{a} \right) }^{2r+1} 
      \frac{(m/n) \cdots (r+m/n)}{(r+1) \cdots (2r+1)} \\
&   & \times {}_{2} F_{1} \left( r+1-m/n, r+1; 2r+2; (a-b)/a \right). 
\end{eqnarray*}

Multiplying the top and bottom of the right-hand side by $q_{r}$, applying the
lower bound for $q_{r}$ in \eqref{eq:thm2-q-est} in Lemma~\ref{lem:thm2-pq-ub},
along with \eqref{eq:LB} and then simplifying, we obtain
\[
\left| q_{r} (a/b)^{m/n} - p_{r} \right|
> \left| \left( \frac{D_{m,n,r}}{N_{d,m,n,r}} \right)^{2} 
  \frac{(m/n) \cdots (r+m/n)}{(r+1) \cdots (2r+1)}
  (a-b)^{2r}(1+\Real(b/a))^{r}
  \frac{a-b}{aq_{r}} \right|.
\]

Recalling the definition of $N_{d,m,n,r}$, that $d=(a-b)^{2}$ and that $X_{m,n,r}(z)$
is a monic polynomial, it follows that $N_{d,m,n,r} \leq (a-b)^{r}$. Using this,
along with Lemma~\ref{lem:denom-LB2}, yields
\[
\left| q_{r} (a/b)^{m/n} - p_{r} \right|
> \left| (1+\Real(b/a))^{r}
  \frac{a-b}{60naq_{r}} \right|.
\]

The desired lower bound now follows since $\Real(b/a) \geq 0$.

For the upper bound, we consider $b/a$ rational with $0<b/a<1$ and
$|b/a|=1$ separately.

To obtain the upper bound when $|b/a|=1$, we start by writing
$b/a=\exp \left( \varphi i \right)$ with $-\pi<\varphi \leq \pi$ and showing
that
\begin{equation}
\label{eq:phi-UB}
|\varphi| \leq 1.22|(a-b)/b|
\end{equation}
for $|b/a-1|<1$.

If $|b/a-1|<1$, then $|\varphi|<1.05$. Using the well-known inequality
$2|\varphi|/\pi \leq |\sin(\varphi)|$ for all $|\varphi| \leq \pi$, we obtain
$|\varphi|<1.22|\sin(\varphi)|=1.22|\Imag(b/a)|$ for $|\varphi|<1.05$.
We can write $b/a=1-(a-b)/a$ and so $\Imag(b/a)=\Imag( 1-(a-b)/a )
=-\Imag( (a-b)/a )$. Since $|\Imag(z)| \leq |z|$ for any complex number, $z$,
equation~\eqref{eq:phi-UB} follows.

We apply Lemma~2.5 of \cite{CV}:
\[
\left| R_{m,n,r}(b/a) \right|
\leq \frac{\Gamma(r+1+m/n)}{r!\Gamma(m/n)} |\varphi| \left| 1-\sqrt{b/a} \right|^{2r},
\]
where $|b/a-1|<1$ and $b/a=\exp \left( \varphi i \right)$. So
\begin{align*}
\left| \left( \frac{a}{b} \right)^{m/n} \frac{a^{r} D_{m,n,r}}{N_{d,m,n,r}} R_{m,n,r}(b/a) \right|
& \leq \frac{a^{r} D_{m,n,r}}{N_{d,m,n,r}} \frac{\Gamma(r+1+m/n)}{r!\Gamma(m/n)} |\varphi| \left| 1-\sqrt{b/a} \right|^{2r} \\
& < 1.22 \left| \frac{a-b}{b} \right| \cC_{n} \left( \frac{\cD_{n}}{\cN_{d,n}} \right)^{r} \left| \sqrt{a}-\sqrt{b} \right|^{2r},
\end{align*}
by \eqref{eq:phi-UB} and using \eqref{eq:cndn-defn}.

To obtain the upper bound when $0<b/a<1$ is rational, we use Pochhammer's integral
(see equation~(1.6.6) of \cite{Sl}). We can write
\[
{}_{2} F_{1} \left( r+1-m/n, r+1; 2r+2; z \right)
=\frac{\Gamma(2r+2)}{\Gamma(r+1)\Gamma(r+1)} \int_{0}^{1} t^{r}(1-t)^{r}(1-zt)^{-r-1+m/n}dt.
\]

If $0<z<1$, then $(1-zt)^{-1+m/n}$ is monotonically increasing as $t$ goes from
$0$ to $1$, since $m/n<1$, so its maximum value occurs at $t=1$. I.e.,
$(1-z)^{-1+m/n} \leq (1-z)^{-1+m/n}$. Here $z=1-b/a$, so this maximum is
$(b/a)^{-1+m/n}=(a/b)^{1-m/n}$.

Also, the function $t(1-t)(1-t(a-b)/a)^{-1}$ takes its maximum value at
$t=\sqrt{a}/\left( \sqrt{a}+\sqrt{b} \right)$, where it takes the value
$a/ \left( \sqrt{a}+\sqrt{b} \right)^{2}$.

Hence
\[
\left| \int_{0}^{1} t^{r}(1-t)^{r}(1-zt)^{-r-1+m/n}dt \right|  
\leq (a/b)^{1-m/n} \left\{ a { \left( a^{1/2}+b^{1/2} \right) }^{-2} \right\}^{r}
\]
and so
\begin{eqnarray*}
\left| q_{r} (a/b)^{m/n} - p_{r} \right|
& \leq & a^{r} \frac{D_{m,n,r}}{N_{d,m,n,r}} { \left( \frac{a}{b} \right) }^{m/n} 
         { \left( \frac{a-b}{a} \right) }^{2r+1}
      \frac{\Gamma(r+1+m/n) \Gamma(r+1)}{\Gamma(2r+2)\Gamma(m/n)} \\
&      & \times \frac{\Gamma(2r+2)}{\Gamma(r+1)\Gamma(r+1)}
         (a/b)^{1-m/n} \left\{ a { \left( a^{1/2}+b^{1/2} \right) }^{-2} \right\}^{r} \\
&   =  & \frac{D_{m,n,r}}{N_{d,m,n,r}}
         { \left( \frac{a-b}{b} \right) }
         \frac{\Gamma(r+1+m/n)}{\Gamma(r+1)\Gamma(m/n)} \left( a^{1/2}-b^{1/2} \right)^{2r} \\
&   =  & \cC_{n} \left( \frac{\cD_{n}}{\cN_{d,n}} \right)^{r}
         \left( \frac{a-b}{b} \right) \left( a^{1/2}-b^{1/2} \right)^{2r},
\end{eqnarray*}
after simplifying and using \eqref{eq:cndn-defn}.
\end{proof}

\section{Proof of Theorem~\ref{thm:gen-hypg2}}

By the lower bound in Lemma~\ref{lem:thm2-8}, we can take $c$ in Theorem~\ref{thm:gen-hypg2}
to be $60n|a|$ when $p/q=p_{i}/q_{i}$ for some non-negative integer $i$. So we
need only prove Theorem~\ref{thm:gen-hypg2} for those numbers $p/q \neq 
p_{i}/q_{i}$ for any non-negative integer $i$. 

All that is required is a simple application of Lemma~\ref{lem:approx} using
Lemmas~\ref{lem:thm2-pq-ub} and \ref{lem:thm2-8} to provide the values of
$k_{0},\ell_{0},E$ and $Q$. 

From these last two lemmas, we can choose $k_{0} = 1.072\cC_{n}$, 
$E = \left( \cN_{d,n}/\cD_{n} \right) \left| a^{1/2} - b^{1/2} \right|^{-2}$ 
and $Q = \left( \cD_{n}/\cN_{d,n} \right) \left| a^{1/2} + b^{1/2} \right|^{2}$.

From equation~\eqref{eq:thm2-remest} in Lemma~\ref{lem:thm2-8}, an obvious choice for
$\ell_{0}$ would be $\ell_{0} = 1.22|(a-b)/b|\cC_{n}$. However, with this choice
$2\ell_{0}E<1$ if $|a-b|$ is not small, so the conditions in Lemma~\ref{lem:approx}
are not satisfied. We do not want to increase the size of $2\ell_{0}E$ too much,
as otherwise the dependence of $c$ and $a$ will increase. So we will choose
$\ell_{0}$ proportional to $1/E$, namely $\ell_{0}=c_{1}\cC_{n}\cD_{n}\left| a^{1/2} - b^{1/2} \right|^{2}$
for some absolute constant $c_{1}>1$ such that
$\ell_{0}>1.22|(a-b)/b|\cC_{n}$. With such a choice, the condition
$\left| q_{r}\theta-p_{r} \right| \leq \ell_{0}E^{-r}$ in Lemma~\ref{lem:approx}
will hold and $2\ell_{0}E=c_{1}n\mu_{n}\cC_{n} \geq 1$ will also hold.

We first determine $c_{2} \geq 1$ such that
\[
\frac{|a-b|}{b} \leq c_{2} \left| a^{1/2} - b^{1/2} \right|^{2}
=c_{2} |b| \left| 1-(a/b)^{1/2} \right|^{2}
=c_{2} |b| \left| 1-(1+(a-b)/b)^{1/2} \right|^{2}.
\]

If $a=a_{R}+a_{I}i$, then $(a-b)/b=2a_{I}i/ \left( a_{R}-a_{I}i \right)$, so
$|(a-b)/b| \leq 2$.
Using the series expansion of $\sqrt{1+z}$, we find that
$\left| 1-\sqrt{1+z} \right| \geq \left( \sqrt{3}-1 \right)|z|/2$ for all
$|z| \leq 2$. Applying this inequality with $z=(a-b)/b$, we have
\[
\left| 1-(1+(a-b)/b)^{1/2} \right| \geq \left( \sqrt{3}-1 \right) |(a-b)/b|/2.
\]

So $\left| a^{1/2} - b^{1/2} \right|^{2}
\geq \left( \sqrt{3}-1 \right)^{2} \left| (a-b)^{2}/(4b) \right|
\geq \left( \sqrt{3}-1 \right)^{2} \left| (a-b)/(4b) \right|$.

Hence we can take $c_{2}=2\left( 2+\sqrt{3} \right)$ and put
\[
\ell_{0} = 9.2 \cC_{n} \cD_{n} \left| a^{1/2} - b^{1/2} \right|^{2},
\]
so that $\ell_{0}>1.22|(a-b)/b|\cC_{n}$.
Also,
\[
2\ell_{0}E=2 \cdot 9.2 \cC_{n} \cD_{n} \left| a^{1/2} - b^{1/2} \right|^{2}
\left( \cN_{d,n}/\cD_{n} \right) \left| a^{1/2} - b^{1/2} \right|^{-2}
=18.4\cC_{n}\cN_{d,n} \geq 18.4.
\]

Lemma~\ref{lem:thm2-pq-ub} ensures that 
$p_{r}q_{r+1} \neq p_{r+1}q_{r}$. In addition, as we saw above, $\cN_{d,n} \leq |a-b|$
and $\cD_{n} \geq 1$, so $Q \geq \left| a^{1/2} + b^{1/2} \right|^{2}/|a-b| > 1$ 
and $\ell_{0} > 0$ since $a \neq b$.
If $E > 1$, then we can use Lemma~\ref{lem:approx}. 

Lastly, we consider the quantity $c$ in Lemma~\ref{lem:approx}. Using the above
expressions for $2\ell_{0}E$, we can write it
as
\[
2.15\cC_{n}\left( 18.4 \cC_{n}\cN_{d,n} \right)^{\kappa}
\leq 3|a|\cC_{n}\left( 20 \cC_{n}\cN_{d,n} \right)^{\kappa},
\]
since $\cN_{d,n} \leq n\mu_{n}$. Noting that $\kappa>1$, we see that this is
also larger than $60n|a|$, completing the proof of Theorem~\ref{thm:gen-hypg2}.

\section{Proof of Theorem~\ref{thm:dn-UB}}
\label{sect:thm2-proof}

For $n=3$, we use the bounds already established in Lemma~5.1(b) of \cite{Vout1}.

The first two subsections of this section provide the proof of parts~(a) and (b)
of the theorem. We proceed in two steps for each $4 \leq n \leq 1009$.\\
(1) we determine $r_{\text{comp}}$, such that for all $m$
and all $r \geq r_{\text{comp}}$, we can use more-or-less analytic techniques to
show that our choice of $\cD_{n}$ works with $\cC_{n}=1$. So in this step, we
also determine the value of $\cD_{n}$ we will use.\\
(2) for all $m$ and all
$r<r_{\text{comp}}$, we essentially calculate directly the quantity within the
outer $\max$ on the left-hand side of equation~\eqref{eq:cndn-defn}. It is in this way
that we determine the value of $\cC_{n}$ that we need.

We prove part~(c) in the last subsection of this section.

\subsection{Determining $\cD_{n}$ and $r_{\text{comp}}$}
\label{subsect:cDn}

We shall use estimates for each of the quantities on the left-hand side of
\eqref{eq:cndn-defn}: the $\Gamma$ function quantities, the numerator and the
denominator.

To estimate the denominator, $D_{m,n,r}$, we divide the prime divisors of $D_{m,n,r}$
into two sets, according to their size. We let $D_{m,n,r}^{(S)}$ denote the
contribution to $D_{m,n,r}$ from primes at most $(nr)^{1/2}$ and let $D_{m,n,r}^{(L)}$
denote the contribution from the remaining, larger, primes. 

\subsubsection{Numerator upper bounds}

Put $d_{1}=\gcd \left( d, n^{2} \right)$ and $d_{2}=\gcd \left( d/d_{1}, n^{2} \right)$,
as in \cite{Vout4}.
By Lemma~6 of \cite{Vout4}, we have
\[
\frac{\cN_{d,n}^{r}}{N_{d,m,n,r}}
\leq \frac{\prod_{p \mid n} p^{r\min( v_{p}(d)/2, v_{p}(n)+1/(p-1))}}
{d_{1}^{\lfloor r/2 \rfloor}
\prod_{p|d_{2}} p^{\min ( \lfloor v_{p}(d_{2})r/2 \rfloor, v_{p}(r!))}}.
\]

We examine the terms in the products on the right-hand side and consider three
possibilities.

(i) If $v_{p}(n)>0$ and $v_{p} \left( d_{2} \right)=0$, then $v_{p}(d)/2 \leq v_{p}(n)$,
so
\[
p^{r\min( v_{p}(d)/2, v_{p}(n)+1/(p-1))}
=p^{rv_{p}(d)/2}=p^{rv_{p}(d_{1})/2}.
\]

(ii) If $p \geq 3$ and $p \mid d_{2}$, or if $p=2$ and
$v_{2}\left( d_{2} \right) \geq 2$, then
$p^{\min( v_{p}(d)/2, v_{p}(n)+1/(p-1))}=p^{v_{p}(n)+1/(p-1)}$.
Furthermore, if $p \geq 3$ and $p \mid d_{2}$, then
\[
v_{p}(r!) \geq \min \left( \lfloor v_{p}\left( d_{2} \right)r/2 \rfloor, v_{p}(r!) \right)
\geq \min \left( \lfloor r/2 \rfloor, v_{p}(r!) \right)=v_{p}(r!),
\]
since
$v_{p}(r!) \leq r/(p-1)$. Similarly, if $p=2$ and
$v_{2}\left( d_{2} \right) \geq 2$, then
\[
v_{p}(r!) \geq \min \left( \lfloor v_{p}\left( d_{2} \right)r/2 \rfloor, v_{p}(r!) \right)
\geq \min \left( \lfloor r \rfloor, v_{p}(r!) \right)=v_{p}(r!).
\]
Thus
\[
\frac{p^{r\min( v_{p}(d)/2, v_{p}(n)+1/(p-1))}}
{p^{\min ( \lfloor v_{p}(d_{2})r/2 \rfloor, v_{p}(r!))}}
=\frac{p^{r\left( v_{p}(n)+1/(p-1)\right)}}{p^{v_{p}(r!)}}
=\frac{p^{r\left( v_{p}(d_{1})/2+1/(p-1)\right)}}{p^{v_{p}(r!)}},
\]
the last equality holding because $v_{p} \left( d_{1} \right)v_{p} \left( n^{2} \right)$
when $v_{p} \left( d_{2} \right) \geq 1$.

(iii) Lastly, if $p=2$ and $v_{p} \left( d_{2} \right)=1$, then
\[
\min \left( v_{p}(d)/2, v_{p}(n)+1/(p-1) \right)
=\min \left( v_{p}\left( d_{1} \right)+1/2, v_{p}(n)+1 \right).
\]
Since
$v_{p} \left( d_{2} \right)>0$, it follows that
\[
\min \left( v_{p}\left( d_{1} \right)+1/2, v_{p}(n)+1 \right)
=v_{p}\left( d_{1} \right)+1/2.
\]

Also $\min \left( \lfloor v_{p}\left( d_{2} \right)r/2 \rfloor, v_{p}(r!) \right)
=\min \left( \lfloor r/2 \rfloor, v_{p}(r!) \right)
=\lfloor r/2 \rfloor$. So
\[
\frac{p^{r\min( v_{p}(d)/2, v_{p}(n)+1/(p-1))}}
{p^{\min ( \lfloor v_{p}(d_{2})r/2 \rfloor, v_{p}(r!))}}
= 2^{rv_{2}(d_{1})/2} 2^{r/2-\lfloor r/2 \rfloor}
\leq 2^{rv_{2}(d_{1})/2} 2^{r/(2-1)-v_{2}(r!)}.
\]

So we always have
\begin{equation}
\label{eq:numerbnd-1}
\frac{\cN_{d,n}^{r}}{N_{d,m,n,r}}
\leq 
\frac{d_{1}^{r/2}}
{d_{1}^{\lfloor r/2 \rfloor}}
\prod_{p|d_{2}} p^{r/(p-1)-v_{p}(r!)}.
\end{equation}

For $r \geq 1$, we have
\[
0 \leq r/(p-1)-v_{p}(r!) \leq (\log r)/(\log p)+1/(p-1)
\]
(the worst case being $r=1$). Therefore,
\begin{equation}
\label{eq:numerbnd-2}
\frac{\cN_{d,n}^{r}}{N_{d,m,n,r}} \leq n\mu_{n}r^{\omega(n)},
\end{equation}
where $\omega(n)$ is the number of distinct prime factors of $n$.

At least for $r=1$, there are examples showing that this upper bound is sharp.
For larger $r$, it can also be not bad.

\subsubsection{$\Gamma$-term upper bounds}

When considering the $\Gamma$-term estimates in the proof of Lemma~7.4(c0 of
\cite{Vout2}, we showed that
\[
\max \left( 1, \frac{\Gamma(1-m/n) \, r!}{\Gamma(r+1-m/n)},
      \frac{n\Gamma(r+1+m/n)}{m\Gamma(m/n)r!} \right) \frac{\cN_{d,n}^{r}}{N_{d,m,n,r}}
< \frac{n}{n-m} e^{m^{2}/n^{2}}r^{m/n},
\]
for $n \geq 2$. Since $m<n/2$, we have
\begin{equation}
\label{eq:fact}
\max \left( 1, \frac{\Gamma(1-m/n) \, r!}{\Gamma(r+1-m/n)},
      \frac{n\Gamma(r+1+m/n)}{m\Gamma(m/n)r!} \right) \frac{\cN_{d,n}^{r}}{N_{d,m,n,r}}
< (n/2) e^{1/4}r^{1/2}.
\end{equation}

\subsubsection{$D_{m,n,r}^{(S)}$ upper bounds}
\label{subsubsect:DmnrS-UB}

From Lemma~3.3(a) of \cite{Vout1}, we know that 
\[
D_{m,n,r}^{(S)} \leq \prod_{(nr)^{1/3}< p \leq (nr)^{1/2}} p^{2}
\prod_{(nr)^{1/4} < p \leq (nr)^{1/3}} p^{3}
\prod_{p \leq (nr)^{1/4}} p^{\lfloor \log(nr)/(\log(p)) \rfloor}. 
\]
So
\[
\log D_{m,n,r}^{(S)} \leq 2\theta \left( (nr)^{1/2} \right)
+\theta \left( (nr)^{1/3} \right)-3\theta \left( (nr)^{1/4} \right)
+\sum_{p \leq (nr)^{1/4}} \lfloor \log(nr)/(\log(p)) \rfloor \log(p). 
\]

Now $\lfloor x \rfloor \leq 4 \lfloor x/4 \rfloor +3$, so
\[
\sum_{p \leq (nr)^{1/4}} \lfloor \log(nr)/(\log(p)) \rfloor \log(p)
\leq 4\psi \left( (nr)^{1/4} \right) + 3\theta \left( (nr)^{1/4} \right),
\]
where $\displaystyle\theta(x)=\sum_{\stackrel{p \leq x}{p, \text{ prime}}} \log(p)$ and
$\displaystyle\psi(x)=\sum_{\stackrel{p^{n} \leq x}{p, \text{ prime}}} \log(p)$.

Thus,
\begin{eqnarray}
\label{eq:dnrs}
D_{m,n,r}^{(S)}
& \leq & \exp \left\{ 2\theta \left( (nr)^{1/2} \right)
+\theta \left( (nr)^{1/3} \right) +4\psi \left( (nr)^{1/4} \right) \right\} \nonumber \\
&  <   & \exp \left\{ 2.033(nr)^{1/2}+1.017(nr)^{1/3}+4.156(nr)^{1/4} \right\},
\end{eqnarray}
from Theorems~9 and 12 of \cite{RS}.

From equations \eqref{eq:numerbnd-2}, \eqref{eq:fact} and \eqref{eq:dnrs}, we
obtain
\begin{eqnarray}
\label{eq:small}
& & \max \left( 1, \frac{\Gamma(1-m/n) \, r!}{\Gamma(r+1-m/n)},
\frac{n\Gamma(r+1+m/n)}{m\Gamma(m/n)r!} \right) \frac{\cN_{d,n}^{r}}{N_{d,m,n,r}}
D_{m,n,r}^{(S)} \nonumber \\
& < & 0.65n^{2}\mu_{n}r^{\omega(n)+1/2}\exp \left\{ 2.033(nr)^{1/2}+1.017(nr)^{1/3}+4.156(nr)^{1/4} \right\}.
\end{eqnarray}
For convenience in what follows, we will denote this last quantity as $S(n,r)$.

\subsubsection{Upper and lower bounds for $\theta(x;n,k)$}

To obtain an upper bound for $D_{m,n,r}^{(L)}$, we need upper and lower bounds
for $\theta \left( x; n, k \right)$. We want bounds
of the form $\epsilon_{x,n,k}^{(L)} x < \theta(x;n,k)-x/\varphi(n) <\epsilon_{x,n,k}^{(U)} x$.
For this, we use the results in \cite{Ben2} and some computation.

Combining equations~\eqref{eq:theta-UB} and \eqref{eq:theta-UBsqrt} in
Lemma~\ref{lem:bennett}, we find that
the upper bound for
$\left| \theta(x;n,k) - x/\varphi(n) \right|$ in \eqref{eq:theta-UB} holds for
$x \geq 1.8 \cdot 10^{9}$ when $101 \leq n \leq 1009$ and for $x \geq 2.1 \cdot 10^{9}$
when $4 \leq n \leq 100$.

For each $4 \leq n \leq 1009$, we compute $\theta(x;n,k)$ for all $1 \leq k<n$ with
$\gcd(k,n)=1$ and for all $x \leq 2.1 \cdot 10^{9}$ to find the last value of
$x$, $X_{n}$, that breaches \eqref{eq:theta-UB} for any $1 \leq k<n$ with
$\gcd(k,n)=1$.

However, these $X_{n}$'s are still quite large (e.g., $X_{4}=1,472,117,809$),
which means that $r$ would have to be quite large for \eqref{eq:theta-UB} to give
a good upper bound for $D_{m,n,r}^{(L)}$. So we break the interval $\left[ 1, X_{n}+2000 \right]$
into $\lfloor X_{n}/2000 \rfloor +2$ subintervals of size $2000$, $I_{i}=[2000(i-1)+1,2000i]$,
and compute to obtain values $\epsilon_{LB,i}$ and
$\epsilon_{UB,i}$ such that if $x \geq 2000(i-1)+1$, then
$\epsilon_{LB,i} x < \theta(x; n,k)-x/\varphi(n) < \epsilon_{UB,i} x$ for all
$1 \leq k \leq n$ with $\gcd(k,n)=1$.

For any positive real number $x$, let $i$ be the largest positive integer such
that $x \geq 2000(i-1)+1$. we will let $\theta_{UB}(x;n)=x/\varphi(n)+\epsilon_{UB,i}x$
and $\theta_{LB}(x;n)=x/\varphi(n)-\epsilon_{LB,i}x$. Note that
$\theta_{UB}(x;n) > \theta(x;n,k) > \theta_{LB}(x;n)$ for all
$1 \leq k \leq n$ with $\gcd(k,n)=1$. This notation will
be convenient for us in what follows.

\subsubsection{$D_{m,n,r}^{(L)}$ upper bounds}

From Lemma~3.3(b) of \cite{Vout1}, we see that for any positive integer $N$
satisfying $nr/(nN+n/2) \geq (nr)^{1/2}$, we have 
\begin{eqnarray}
\label{eq:DL-defn}
D_{m,n,r}^{(L)} 
& \leq & \exp \left\{ \sum_{A=0}^{N-1} \sum_{\ell=1,(\ell,n)=1}^{n/2} \left( \theta(nr/(nA+\ell);n,k_{\ell})  
		      - \theta(nr/(nA+n-\ell);n,k_{\ell}) \right) \right\} \nonumber \\
& & \times \exp \left\{ \sum_{\ell=1,(\ell,n)=1}^{n/2} \theta(nr/(nN+\ell);n, k_{\ell}) \right\} \nonumber \\
& < & \exp \left\{ \sum_{A=0}^{N-1} \sum_{\ell=1,(\ell,n)=1}^{n/2} \left( \theta_{UB}(nr/(nA+\ell);n)  
		      - \theta_{LB}(nr/(nA+n-\ell);n) \right) \right\} \\
& & \times \exp \left\{ \sum_{\ell=1,(\ell,n)=1}^{n/2} \theta_{UB}(nr/(nN+\ell);n) \right\}, \nonumber
\end{eqnarray}
where $k_{\ell} \equiv (-m)\ell^{-1} \bmod n$. We will denote the last quantity
as $D^{(L)}(N,n,r)$.

\subsubsection{Combining the bounds}

Combining equations \eqref{eq:small} and \eqref{eq:DL-defn}, we find that the
left-hand side of \eqref{eq:cndn-defn} is less than $S(n,r)D^{(L)}(N,n,r)$.

Incrementing $r$ in steps of size $100,000$ and checking positive integers $N$
up to $200$,
we determined $\log \left( S(n,r)D^{(L)}(N,n,r) \right)$ for each pair
$(r,N)$ and then chose the values of $r$ and $N$ (we denote this $r$ by
$r_{{\rm comp}}$) such that
$\log \left( S(n,r)D^{(L)}(N,n,r) \right)/r$ is as small as possible to obtain
an upper bound for left-hand side of \eqref{eq:cndn-defn}
once
$r \geq r_{{\rm comp}}$. For $n \geq 223$, we also cap $r_{{\rm comp}}$ by $10^{7}$
to make the computations more feasible.
This is the value we will use for $\log \cD_{n}$.
E.g., for $n=4$, $N=90$ and $r_{\rm comp}=39,900,000$, this
suggests using $\log \cD_{3}=1.58$.

We now know $\cD_{n}$ as well as how much computation is required to establish
our desired inequalities for all $r \geq 0$ (a computation which will yield
$\cC_{n}$), so we are ready to describe the required computations.

\subsection{Determining $\cC_{n}$ and checking $r<r_{\text{comp}}$}

For each pair $(m,n)$ with $1 \leq m<n/2$, $4 \leq n \leq 1009$ and $\gcd(m,n)=1$,
we take the following steps for each $0 \leq r < r_{{\rm comp}}$.

(1) We compute directly the $\Gamma$ terms in \eqref{eq:cndn-defn}, noting that
the value for $r$ can be computed from the value for $r+1$.

(2) We initially estimate the numerator, $\cN_{d,n}^{r}/N_{d,m,n,r}$ in fact,
using \eqref{eq:numerbnd-1}, where we
bound $d_{1}^{r/2-\lfloor r/2 \rfloor}$ from above by $n$ and take the product
over all primes dividing $n$, rather than $d_{2}$.

This is much faster than calculating the maximum possible value of
$\cN_{d,n}^{r}/N_{d,m,n,r}$ precisely over all values of $d$. However, if, for a
particular value of $r$, after the denominator steps that follow, this estimate
leads to a large value of $\cC_{n}$, then we do calculate $\cN_{d,n}^{r}/N_{d,m,n,r}$
more precisely using the expression for $X_{m,n,r} \left( 1-\sqrt{d}\, x \right)$
in terms of $d_{1}$, $d_{2}$ and $d_{3}$ in the proof of Lemma~6 in \cite{Vout4}.

(3) we initially use the upper bound
\[
D_{m,n,r}^{(S)} \leq \prod_{p \leq (nr)^{1/2}} p^{\lfloor \log(nr)/(\log(p)) \rfloor},
\]
which holds by Lemma~3.3(a) of \cite{Vout1}. We calculate the right-hand side
directly for each value of $m$, $n$ and $r$.

As in step~(2), if this upper bound leads to a large value of $\cC_{n}$, then
we calculate $D^{(S)}_{m,n,r}$ directly using Proposition~3.2 of \cite{Vout1}.

(4) we compute $D^{(L)}_{m,n,r}$ exactly using the same technique as in \cite{Vout1}
(see Step~(5) of the proof of Lemma~5.1(b) there) of using Lemma~3.3(b) there
and calculating the contributions from each interval and congruence class via the
endpoints of these intervals. The only difference is that here we grow what is
called $A(r)$ in \cite{Vout1} over the course of the calculation so
that $A(r)$ is the largest integer such that $nr/(nA(r)+n-\ell)>\sqrt{nr}$.

In this manner, for all $r<r_{{\rm comp}}$, we estimate the left-hand side of
\eqref{eq:cndn-defn} and hence find a value of $\cC_{n}$ that would work with
the value of $\cD_{n}$. If for any such $r$, the value of $\cC_{n}$ exceeds the
value of $\cC_{n}$ found for smaller values of $r$, then we use the more precise
methods for bounding $\cN_{d,n}^{r}/N_{d,m,n,r}$ and $D^{(S)}(m,n,r)$ described
in steps~(2) and (3) above to get a more precise upper bound for $\cC_{n}$. So
the maximum value of $\cC_{n}$ obtained in this way is the one that we use.

As part of these calculations, we also determined $\cD_{2,n}$ in Tables~\ref{table:1}
through \ref{table:7}.

All these calculations were performed using code written in the Java programming
language (JDK 16) and run on a Windows laptop with an Intel i7-9750H 2.60GHz CPU.
Unsurprisingly, the amount of time required for each value of $n$
increased with $n$. For example, for $n=229$, $2,175$ seconds of CPU time was used,
whereas for $n=1009$, the CPU time was $16,643$ seconds. The code is
available upon request.

\subsection{Proof of Theorem~\ref{thm:dn-UB}(c)}

It was shown in the proof of Lemma~7.4(d) of \cite{Vout2} that
\[
\max \left( 1, \frac{\Gamma(1-m/n) \, r!}{\Gamma(r+1-m/n)},
\frac{n\Gamma(r+1+m/n)}{m\Gamma(m/n)r!} \right) D_{m,n,r}
\leq n^{r} \mu_{n}^{r}.
\]

Applying equation~\eqref{eq:numerbnd-1}, if $d_{2}=1$, then part~(c) follows as
$d_{1} \leq n^{2}$.

\section{Thue's Fundamentaltheorem}

The initial hope for this work was to improve the constant not just for the
usual hypergeometric method, but for Thue's Fundamentaltheorem too (e.g.,
Theorem~1 in \cite{Vout4}, as well as the theorems in \cite{Vout2,Vout3}).

The two key parts of Thue's Fundamentaltheorem are the following.

(1) let $t$ be a rational integer which is not a perfect square and put
$\bbK = \bbQ \left( \sqrt{t} \right)$. Suppose that $\eta \in \cO_{\bbK}$
and that $\sigma$ is the non-trivial element of
$\Gal \left( \bbK/\bbQ \right)$. Then
$\sigma(\eta)^{r} X_{m,n,r} \left( \eta / \sigma(\eta) \right)$ and
$\sigma(\eta)^{r} Y_{m,n,r} \left( \eta / \sigma(\eta) \right)$
are algebraic conjugates in $\bbK$.

(2) the classical observation (see Lemma~\ref{lem:relation}) that
\[
\left( \eta/\sigma(\eta) \right)^{m/n} Y_{m,n,r}\left( \eta/\sigma(\eta) \right)
- X_{m,n,r} \left( \eta/\sigma(\eta) \right)
= \left( \eta/\sigma(\eta)-1 \right)^{2r+1} R_{m,n,r} \left( \eta/\sigma(\eta) \right),
\]
for $\left| \eta/\sigma(\eta)-1 \right|<1$.

\vspace*{3.0mm}

Due to (1), we can write
\[
\left( \eta/\sigma(\eta) \right)^{m/n} \sigma \left( \sigma(\eta)^{r}X_{m,n,r} \left( \eta/\sigma(\eta) \right) \right)
- \sigma(\eta)^{r}X_{m,n,r} \left( \eta/\sigma(\eta) \right)
= \left( \eta/\sigma(\eta)-1 \right)^{2r+1} \sigma(\eta)^{r}R_{m,n,r} \left( \eta/\sigma(\eta) \right),
\]

To simplify our notation, we will write $q_{r}=\sigma \left( \sigma(\eta)^{r}X_{m,n,r} \left( \eta/\sigma(\eta) \right) \right)$.

Let $\beta$ and $\gamma$ be two distinct non-rational algebraic integers in $\bbK$
and put
\[
\alpha = \frac{\beta+\sigma(\beta)\left( \eta/\sigma(\eta) \right)^{m/n}}
{\gamma+\sigma(\gamma)\left( \eta/\sigma(\eta) \right)^{m/n}}.
\]

Using the idea from \cite{Vout3}, we can write
\[
\left( \gamma q_{r} + \sigma(\gamma)\sigma \left( q_{r} \right) \right) \alpha
- \left( q_{r} \beta + \sigma(\beta)\sigma \left( q_{r} \right) \right)
= \left( \sigma(\beta) - \alpha \sigma(\gamma) \right) \left( \eta/\sigma(\eta)-1 \right)^{2r+1} \sigma(\eta)^{r}R_{m,n,r} \left( \eta/\sigma(\eta) \right).
\]

This gives us a sequence of good approximations to $\alpha$ from our sequence
of good approximations to $\left( \eta/\sigma(\eta) \right)^{m/n}$.

As above with the usual hypergeometric method, to get improved constants
we need a lower bound for 
$\left| \gamma q_{r} + \sigma(\gamma)\sigma \left( q_{r} \right) \right|$.
Notice that if $\gamma q_{r}=a_{r}+b_{r}\sqrt{t}$, then
$\gamma q_{r} + \sigma(\gamma)\sigma \left( q_{r} \right)=2a_{r}$.
How can we bound $\left| 2a_{r} \right|$ from below?

Unfortunately, it is easy to compute examples with the real
parts of the values of the above hypergeometric functions having sign changes on the
unit circle that get closer to $1$ as $r$ gets larger. This seems to suggest that
our approach here will not provide better constants for Thue's Fundamentaltheorem.
But perhaps it is only some fresh ideas that are required.

\section*{Acknowledgements}

Some of the ideas in this paper resulted from discussions during the
``Transcendence and Diophantine Problems'' conference celebrating 100 years since
Feldman's birth, held at MIPT, Moscow, 10--14 June 2019. The author thanks Professor
Moshchevitin for the invitation, as well as for his generosity and assistance during
this conference, and for providing a stimulating environment.

\appendix
\section{Values of $\cC_{n}$, $\cD_{n}$ and supporting data}

In the following tables, we provide the values of $\cC_{1,n}$, $\log \cD_{1,n}$ and
$\log \cD_{2,n}$ for parts~(a) and (b) of Theorem~\ref{thm:dn-UB}. We also provide
information about the calculations used in the proof of this theorem, as described
in Section~\ref{sect:thm2-proof}.

Here is a description of the other fields in these tables.\\
$\bullet$ $m_{1,\text{max}}$: the value of $m$ where the maximum value of $\cC_{1,n}$ occurred.\\
$\bullet$ $\log \cD_{\text{Chud},n}$: the value of $\log \cD_{n}$ for Chudnovsky's asymptotic estimate.\\
$\bullet$ $\log n\mu_{n}$: the value of $\log \cD_{n}$ used by Baker and defined in Theorem~\ref{thm:dn-UB}.\\
These two values are provided for comparison with our own values.\\
$\bullet$ $r_{1,\text{max}}$: the value of $r$ where the maximum value of $\cC_{1,n}$ occurred.\\
$\bullet$ $m_{2,\text{max}}$: the value of $m$ where $\cC_{n}=100$ with $\cD_{n}=\cD_{2,n}$ occurred.\\
$\bullet$ $r_{2,\text{max}}$: the value of $r$ where $\cC_{n}=100$ with $\cD_{n}=\cD_{2,n}$ occurred.\\
$\bullet$ $r_{\text{comp}}$: defined at the start of Section~\ref{sect:thm2-proof}.

Note that $m_{1,\text{max}}$ and $m_{2,\text{max}}$ are not included in Tables~\ref{table:4}--\ref{table:7},
since we only consider $m=1$ for such values of $n$.

In some cases, especially for large $n$, $\cC_{n}=100$ is never attained with the
values of $\cD_{n}$ that we can use. In these cases, the entries of $m_{2,\text{max}}$
and $r_{2,\text{max}}$ in the tables are ``--''.

\begin{table}[h!]
\centering
\scalebox{0.9}{%
\begin{tabular}{|c|c|c|c|c|c|c|c|c|c|c|c|}  \hline
 $n$ & $\cC_{1,n}$ & $\log \cD_{\text{Chud},n}$ & $\log \cD_{1,n}$ & $\log \cD_{2,n}$ & $\log n\mu_{n}$ & $m_{1,\text{max}}$ & $m_{2,\text{max}}$ & $r_{1,\text{max}}$ & $r_{2,\text{max}}$ & $r_{\text{comp}}$ & $N$ \\ \hline \hline
3 & $2 \cdot 10^{14}$ & 0.907 & 0.916 & 0.953 & 1.648 & 1 & 1 & 19,946 & 66 & $200 \cdot 10^{6}$ & 201 \\ \hline
4 & $3 \cdot 10^{26}$ & 1.571 & 1.579 & 1.635 & 2.080 & 1 & 1 & 14,983 & 165 & $50 \cdot 10^{6}$ & 99 \\ \hline
5 & $10^{45}$ & 1.337 & 1.348 & 1.410 & 2.012 & 1 & 2 & 7060 & 200 & $45 \cdot 10^{6}$ & 77 \\ \hline
6 & $7 \cdot 10^{24}$ & 2.721 & 2.729 & 2.761 & 3.035 & 1 & 1 & 9912 & 271 & $36 \cdot 10^{6}$ & 65 \\ \hline
7 & $10^{26}$ & 1.625 & 1.638 & 1.716 & 2.271 & 1 & 2 & 12364 & 293 & $47 \cdot 10^{6}$ & 65 \\ \hline
8 & $8 \cdot 10^{20}$ & 2.222 & 2.235 & 2.348 & 2.773 & 1 & 1 & 3529 & 61 & $41 \cdot 10^{6}$ & 57 \\ \hline
9 & $5 \cdot 10^{31}$ & 2.155 & 2.169 & 2.288 & 2.747 & 1 & 2 & 13,953 & 52 & $40 \cdot 10^{6}$ & 58 \\ \hline
10 & $2 \cdot 10^{26}$ & 2.988 & 2.999 & 3.064 & 3.399 & 1 & 1 & 2383 & 107 & $41 \cdot 10^{6}$ & 52 \\ \hline
11 & $7 \cdot 10^{23}$ & 2.020 & 2.038 & 2.158 & 2.638 & 5 & 5 & 2161 & 114 & $44 \cdot 10^{6}$ & 45 \\ \hline
12 & $3 \cdot 10^{32}$ & 3.142 & 3.155 & 3.258 & 3.728 & 5 & 1 & 3568 & 42 & $48 \cdot 10^{6}$ & 56 \\ \hline
13 & $4 \cdot 10^{24}$ & 2.169 & 2.189 & 2.314 & 2.779 & 4 & 4 & 3234 & 30 & $46 \cdot 10^{6}$ & 38 \\ \hline
14 & $2 \cdot 10^{30}$ & 3.203 & 3.216 & 3.350 & 3.657 & 3 & 3 & 2794 & 47 & $46 \cdot 10^{6}$ & 55 \\ \hline
15 & $7 \cdot 10^{30}$ & 3.125 & 3.141 & 3.283 & 3.660 & 4 & 7 & 12515 & 61 & $46 \cdot 10^{6}$ & 45 \\ \hline
16 & $3 \cdot 10^{51}$ & 2.903 & 2.920 & 3.061 & 3.466 & 3 & 3 & 7759 & 55 & $49 \cdot 10^{6}$ & 48 \\ \hline
17 & $4 \cdot 10^{22}$ & 2.410 & 2.435 & 2.576 & 3.011 & 4 & 8 & 2424 & 23 & $50 \cdot 10^{6}$ & 35 \\ \hline
18 & $3 \cdot 10^{26}$ & 3.600 & 3.613 & 3.713 & 4.133 & 1 & 5 & 1553 & 113 & $49 \cdot 10^{6}$ & 59 \\ \hline
19 & $4 \cdot 10^{20}$ & 2.511 & 2.538 & 2.741 & 3.109 & 1 & 6 & 2806 & 73 & $48 \cdot 10^{6}$ & 28 \\ \hline
20 & $5 \cdot 10^{23}$ & 3.513 & 3.530 & 3.666 & 4.092 & 3 & 1 & 4061 & 36 & $45 \cdot 10^{6}$ & 43 \\ \hline
21 & $4 \cdot 10^{32}$ & 3.375 & 3.395 & 3.527 & 3.919 & 4 & 8 & 1507 & 183 & $45 \cdot 10^{6}$ & 35 \\ \hline
22 & $5 \cdot 10^{27}$ & 3.530 & 3.548 & 3.666 & 4.024 & 7 & 7 & 3283 & 107 & $43 \cdot 10^{6}$ & 42 \\ \hline
23 & $7 \cdot 10^{17}$ & 2.687 & 2.715 & 2.908 & 3.279 & 8 & 10 & 1579 & 73 & $49 \cdot 10^{6}$ & 27 \\ \hline
24 & $7 \cdot 10^{37}$ & 3.848 & 3.868 & 3.997 & 4.421 & 11 & 11 & 5920 & 102 & $48 \cdot 10^{6}$ & 37 \\ \hline
25 & $2 \cdot 10^{28}$ & 3.049 & 3.077 & 3.280 & 3.622 & 8 & 8 & 3252 & 52 & $45 \cdot 10^{6}$ & 28 \\ \hline
26 & $8 \cdot 10^{26}$ & 3.660 & 3.680 & 3.792 & 4.165 & 7 & 9 & 1984 & 165 & $49 \cdot 10^{6}$ & 34 \\ \hline
27 & $4 \cdot 10^{19}$ & 3.275 & 3.303 & 3.453 & 3.846 & 8 & 4 & 1251 & 27 & $45 \cdot 10^{6}$ & 28 \\ \hline
28 & $5 \cdot 10^{25}$ & 3.774 & 3.796 & 3.993 & 4.350 & 3 & 13 & 2018 & 38 & $47 \cdot 10^{6}$ & 34 \\ \hline
29 & $8 \cdot 10^{20}$ & 2.901 & 2.936 & 3.185 & 3.488 & 3 & 3 & 601 & 29 & $47 \cdot 10^{6}$ & 22 \\ \hline
30 & $5 \cdot 10^{39}$ & 4.431 & 4.449 & 4.592 & 5.047 & 7 & 7 & 2093 & 102 & $46 \cdot 10^{6}$ & 48 \\ \hline
31 & $4 \cdot 10^{24}$ & 2.963 & 3.000 & 3.216 & 3.549 & 14 & 12 & 1496 & 31 & $50 \cdot 10^{6}$ & 22 \\ \hline
32 & $2 \cdot 10^{27}$ & 3.593 & 3.619 & 3.821 & 4.159 & 7 & 15 & 1231 & 44 & $50 \cdot 10^{6}$ & 29 \\ \hline
33 & $2 \cdot 10^{29}$ & 3.734 & 3.761 & 3.900 & 4.286 & 4 & 8 & 1550 & 23 & $49 \cdot 10^{6}$ & 29 \\ \hline
34 & $4 \cdot 10^{35}$ & 3.877 & 3.903 & 4.013 & 4.397 & 11 & 3 & 2642 & 59 & $47 \cdot 10^{6}$ & 31 \\ \hline
35 & $3 \cdot 10^{27}$ & 3.730 & 3.760 & 3.960 & 4.283 & 1 & 9 & 2470 & 58 & $48 \cdot 10^{6}$ & 26 \\ \hline
36 & $6 \cdot 10^{56}$ & 4.256 & 4.278 & 4.427 & 4.826 & 7 & 17 & 5305 & 16 & $50 \cdot 10^{6}$ & 38 \\ \hline
37 & $5 \cdot 10^{20}$ & 3.129 & 3.169 & 3.352 & 3.712 & 11 & 18 & 1009 & 17 & $50 \cdot 10^{6}$ & 19 \\ \hline
38 & $4 \cdot 10^{16}$ & 3.970 & 3.997 & 4.152 & 4.495 & 15 & 3 & 909 & 67 & $48 \cdot 10^{6}$ & 28 \\ \hline
39 & $5 \cdot 10^{38}$ & 3.873 & 3.904 & 4.064 & 4.427 & 19 & 19 & 6609 & 94 & $47 \cdot 10^{6}$ & 28 \\ \hline
40 & $9 \cdot 10^{39}$ & 4.214 & 4.242 & 4.364 & 4.785 & 19 & 3 & 1809 & 32 & $49 \cdot 10^{6}$ & 28 \\ \hline
41 & $7 \cdot 10^{21}$ & 3.226 & 3.270 & 3.448 & 3.807 & 15 & 5 & 907 & 43 & $47 \cdot 10^{6}$ & 19 \\ \hline
42 & $9 \cdot 10^{35}$ & 4.703 & 4.724 & 4.912 & 5.305 & 19 & 13 & 5452 & 25 & $50 \cdot 10^{6}$ & 40 \\ \hline
43 & $2 \cdot 10^{19}$ & 3.271 & 3.316 & 3.535 & 3.851 & 4 & 10 & 1596 & 45 & $46 \cdot 10^{6}$ & 18 \\ \hline
44 & $3 \cdot 10^{28}$ & 4.145 & 4.175 & 4.316 & 4.718 & 7 & 15 & 4890 & 55 & $46 \cdot 10^{6}$ & 26 \\ \hline
\end{tabular}%
}
\caption{Data for $3 \leq n \leq 44$}
\label{table:1}
\end{table}

\begin{table}[h!]
\centering
\scalebox{0.9}{%
\begin{tabular}{|c|c|c|c|c|c|c|c|c|c|c|c|}  \hline
$n$ & $\cC_{1,n}$ & $\log \cD_{\text{Chud},n}$ & $\log \cD_{1,n}$ & $\log \cD_{2,n}$ & $\log n\mu_{n}$ & $m_{1,\text{max}}$ & $m_{2,\text{max}}$ & $r_{1,\text{max}}$ & $r_{2,\text{max}}$ & $r_{\text{comp}}$ & $N$ \\ \hline \hline
45 & $5 \cdot 10^{20}$ & 4.196 & 4.228 & 4.388 & 4.759 & 22 & 11 & 1480 & 66 & $49 \cdot 10^{6}$ & 24 \\ \hline
46 & $4 \cdot 10^{20}$ & 4.133 & 4.162 & 4.314 & 4.665 & 5 & 19 & 1615 & 68 & $49 \cdot 10^{6}$ & 26 \\ \hline
47 & $4 \cdot 10^{21}$ & 3.355 & 3.402 & 3.589 & 3.934 & 5 & 19 & 1631 & 18 & $47 \cdot 10^{6}$ & 17 \\ \hline
48 & $10^{22}$ & 4.545 & 4.573 & 4.751 & 5.114 & 7 & 7 & 3982 & 32 & $47 \cdot 10^{6}$ & 30 \\ \hline
49 & $2 \cdot 10^{19}$ & 3.647 & 3.691 & 3.849 & 4.217 & 24 & 8 & 688 & 39 & $49 \cdot 10^{6}$ & 20 \\ \hline
50 & $9 \cdot 10^{26}$ & 4.448 & 4.477 & 4.604 & 5.008 & 7 & 7 & 3503 & 102 & $50 \cdot 10^{6}$ & 29 \\ \hline
51 & $2 \cdot 10^{20}$ & 4.101 & 4.139 & 4.345 & 4.659 & 25 & 22 & 1135 & 27 & $50 \cdot 10^{6}$ & 22 \\ \hline
52 & $2 \cdot 10^{27}$ & 4.287 & 4.320 & 4.460 & 4.859 & 3 & 9 & 1712 & 192 & $45 \cdot 10^{6}$ & 26 \\ \hline
53 & $4 \cdot 10^{17}$ & 3.469 & 3.524 & 3.708 & 4.047 & 17 & 6 & 762 & 45 & $42 \cdot 10^{6}$ & 14 \\ \hline
54 & $5 \cdot 10^{20}$ & 4.668 & 4.697 & 4.885 & 5.232 & 11 & 7 & 2062 & 89 & $48 \cdot 10^{6}$ & 28 \\ \hline
55 & $8 \cdot 10^{31}$ & 4.092 & 4.135 & 4.296 & 4.650 & 4 & 14 & 567 & 27 & $47 \cdot 10^{6}$ & 18 \\ \hline
56 & $2 \cdot 10^{27}$ & 4.473 & 4.508 & 4.706 & 5.043 & 1 & 19 & 587 & 105 & $48 \cdot 10^{6}$ & 23 \\ \hline
57 & $6 \cdot 10^{26}$ & 4.198 & 4.240 & 4.459 & 4.756 & 23 & 28 & 1437 & 27 & $48 \cdot 10^{6}$ & 17 \\ \hline
58 & $5 \cdot 10^{27}$ & 4.335 & 4.371 & 4.568 & 4.874 & 17 & 17 & 722 & 33 & $48 \cdot 10^{6}$ & 21 \\ \hline
59 & $2 \cdot 10^{18}$ & 3.571 & 3.630 & 3.957 & 4.148 & 24 & 28 & 655 & 27 & $38 \cdot 10^{6}$ & 13 \\ \hline
60 & $3 \cdot 10^{26}$ & 5.176 & 5.203 & 5.388 & 5.740 & 19 & 11 & 2171 & 96 & $48 \cdot 10^{6}$ & 27 \\ \hline
61 & $5 \cdot 10^{19}$ & 3.603 & 3.664 & 3.876 & 4.180 & 17 & 8 & 1096 & 21 & $36 \cdot 10^{6}$ & 13 \\ \hline
62 & $10^{22}$ & 4.394 & 4.433 & 4.660 & 4.935 & 23 & 23 & 2398 & 31 & $47 \cdot 10^{6}$ & 20 \\ \hline
63 & $3 \cdot 10^{21}$ & 4.453 & 4.494 & 4.723 & 5.017 & 10 & 29 & 589 & 27 & $50 \cdot 10^{6}$ & 21 \\ \hline
64 & $3 \cdot 10^{31}$ & 4.285 & 4.326 & 4.476 & 4.853 & 31 & 9 & 1711 & 47 & $48 \cdot 10^{6}$ & 20 \\ \hline
65 & $3 \cdot 10^{22}$ & 4.232 & 4.281 & 4.505 & 4.791 & 14 & 28 & 677 & 27 & $44 \cdot 10^{6}$ & 17 \\ \hline
66 & $9 \cdot 10^{25}$ & 5.082 & 5.112 & 5.267 & 5.672 & 19 & 29 & 1383 & 35 & $48 \cdot 10^{6}$ & 30 \\ \hline
67 & $4 \cdot 10^{16}$ & 3.693 & 3.759 & 3.923 & 4.269 & 17 & 27 & 635 & 134 & $32 \cdot 10^{6}$ & 13 \\ \hline
68 & $2 \cdot 10^{29}$ & 4.519 & 4.560 & 4.752 & 5.090 & 21 & 31 & 1564 & 67 & $50 \cdot 10^{6}$ & 20 \\ \hline
69 & $4 \cdot 10^{15}$ & 4.366 & 4.412 & 4.623 & 4.926 & 1 & 7 & 707 & 26 & $49 \cdot 10^{6}$ & 19 \\ \hline
70 & $3 \cdot 10^{26}$ & 5.080 & 5.112 & 5.287 & 5.669 & 23 & 17 & 1120 & 31 & $50 \cdot 10^{6}$ & 27 \\ \hline
71 & $2 \cdot 10^{17}$ & 3.749 & 3.819 & 4.073 & 4.324 & 10 & 14 & 1096 & 13 & $30 \cdot 10^{6}$ & 11 \\ \hline
72 & $4 \cdot 10^{27}$ & 4.951 & 4.987 & 5.160 & 5.520 & 31 & 35 & 4221 & 124 & $48 \cdot 10^{6}$ & 21 \\ \hline
73 & $5 \cdot 10^{13}$ & 3.775 & 3.848 & 4.053 & 4.351 & 11 & 4 & 442 & 31 & $30 \cdot 10^{6}$ & 11 \\ \hline
74 & $3 \cdot 10^{20}$ & 4.553 & 4.595 & 4.807 & 5.098 & 27 & 1 & 1549 & 45 & $48 \cdot 10^{6}$ & 18 \\ \hline
75 & $6 \cdot 10^{24}$ & 4.704 & 4.748 & 4.967 & 5.270 & 2 & 8 & 1913 & 58 & $50 \cdot 10^{6}$ & 21 \\ \hline
76 & $9 \cdot 10^{32}$ & 4.617 & 4.659 & 4.874 & 5.188 & 23 & 31 & 446 & 57 & $49 \cdot 10^{6}$ & 18 \\ \hline
77 & $2 \cdot 10^{19}$ & 4.348 & 4.409 & 4.600 & 4.908 & 3 & 15 & 576 & 101 & $36 \cdot 10^{6}$ & 13 \\ \hline
78 & $9 \cdot 10^{23}$ & 5.227 & 5.261 & 5.493 & 5.813 & 19 & 31 & 1841 & 92 & $45 \cdot 10^{6}$ & 28 \\ \hline
79 & $3 \cdot 10^{11}$ & 3.851 & 3.928 & 4.150 & 4.426 & 25 & 12 & 101 & 11 & $28 \cdot 10^{6}$ & 10 \\ \hline
80 & $9 \cdot 10^{28}$ & 4.910 & 4.950 & 5.126 & 5.478 & 3 & 33 & 1571 & 107 & $48 \cdot 10^{6}$ & 20 \\ \hline
81 & $3 \cdot 10^{18}$ & 4.376 & 4.435 & 4.658 & 4.944 & 40 & 23 & 484 & 50 & $40 \cdot 10^{6}$ & 14 \\ \hline
82 & $9 \cdot 10^{18}$ & 4.646 & 4.692 & 4.900 & 5.193 & 35 & 35 & 822 & 37 & $49 \cdot 10^{6}$ & 19 \\ \hline
83 & $3 \cdot 10^{21}$ & 3.899 & 3.980 & 4.172 & 4.473 & 18 & 8 & 765 & 23 & $27 \cdot 10^{6}$ & 10 \\ \hline
84 & $3 \cdot 10^{21}$ & 5.433 & 5.468 & 5.643 & 5.998 & 37 & 23 & 1017 & 94 & $49 \cdot 10^{6}$ & 24 \\ \hline
85 & $4 \cdot 10^{15}$ & 4.461 & 4.527 & 4.713 & 5.023 & 19 & 9 & 568 & 17 & $33 \cdot 10^{6}$ & 12 \\ \hline
86 & $3 \cdot 10^{17}$ & 4.689 & 4.736 & 4.919 & 5.238 & 17 & 11 & 593 & 53 & $50 \cdot 10^{6}$ & 18 \\ \hline
\end{tabular}%
}
\caption{Data for $45 \leq n \leq 86$}
\label{table:2}
\end{table}

\begin{table}[h!]
\centering
\scalebox{0.9}{%
\begin{tabular}{|c|c|c|c|c|c|c|c|c|c|c|c|}  \hline
$n$ & $\cC_{1,n}$ & $\log \cD_{\text{Chud},n}$ & $\log \cD_{1,n}$ & $\log \cD_{2,n}$ & $\log n\mu_{n}$ & $m_{1,\text{max}}$ & $m_{2,\text{max}}$ & $r_{1,\text{max}}$ & $r_{2,\text{max}}$ & $r_{\text{comp}}$ & $N$ \\ \hline \hline
87 & $3 \cdot 10^{18}$ & 4.574 & 4.632 & 4.830 & 5.136 & 35 & 35 & 1398 & 33 & $49 \cdot 10^{6}$ & 14 \\ \hline
88 & $10^{23}$ & 4.842 & 4.888 & 5.115 & 5.411 & 9 & 23 & 1097 & 31 & $49 \cdot 10^{6}$ & 18 \\ \hline
89 & $2 \cdot 10^{12}$ & 3.966 & 4.051 & 4.264 & 4.540 & 37 & 35 & 180 & 25 & $25 \cdot 10^{6}$ & 9 \\ \hline
90 & $5 \cdot 10^{25}$ & 5.571 & 5.605 & 5.775 & 6.145 & 23 & 29 & 590 & 54 & $49 \cdot 10^{6}$ & 24 \\ \hline
91 & $9 \cdot 10^{15}$ & 4.488 & 4.561 & 4.762 & 5.049 & 5 & 29 & 385 & 57 & $29 \cdot 10^{6}$ & 10 \\ \hline
92 & $3 \cdot 10^{21}$ & 4.788 & 4.836 & 5.042 & 5.358 & 35 & 31 & 499 & 45 & $49 \cdot 10^{6}$ & 16 \\ \hline
93 & $10^{17}$ & 4.635 & 4.698 & 4.924 & 5.197 & 37 & 35 & 833 & 27 & $36 \cdot 10^{6}$ & 13 \\ \hline
94 & $2 \cdot 10^{17}$ & 4.770 & 4.821 & 4.961 & 5.321 & 23 & 17 & 2659 & 33 & $46 \cdot 10^{6}$ & 17 \\ \hline
95 & $3 \cdot 10^{11}$ & 4.558 & 4.628 & 4.809 & 5.120 & 41 & 46 & 591 & 42 & $30 \cdot 10^{6}$ & 11 \\ \hline
96 & $2 \cdot 10^{31}$ & 5.239 & 5.281 & 5.462 & 5.807 & 7 & 7 & 1661 & 46 & $48 \cdot 10^{6}$ & 19 \\ \hline
97 & $5 \cdot 10^{14}$ & 4.050 & 4.140 & 4.344 & 4.623 & 45 & 36 & 332 & 17 & $23 \cdot 10^{6}$ & 8 \\ \hline
98 & $9 \cdot 10^{16}$ & 5.038 & 5.085 & 5.339 & 5.603 & 9 & 37 & 375 & 50 & $50 \cdot 10^{6}$ & 19 \\ \hline
99 & $2 \cdot 10^{18}$ & 4.819 & 4.881 & 5.101 & 5.385 & 28 & 32 & 971 & 79 & $35 \cdot 10^{6}$ & 14 \\ \hline
100 & $2 \cdot 10^{23}$ & 5.133 & 5.178 & 5.405 & 5.701 & 23 & 41 & 1587 & 45 & $48 \cdot 10^{6}$ & 18 \\ \hline
\end{tabular}%
}
\caption{Data for $87 \leq n \leq 100$}
\label{table:3}
\end{table}

\begin{table}[h!]
\centering
\begin{tabular}{|c|c|c|c|c|c|c|c|c|c|}  \hline
$n$ & $\cC_{1,n}$ & $\log \cD_{\text{Chud},n}$ & $\log \cD_{1,n}$ & $\log \cD_{2,n}$ & $\log n\mu_{n}$ & $r_{1,\text{max}}$ & $r_{2,\text{max}}$ & $r_{\text{comp}}$ & $N$ \\ \hline \hline
101 & $3 \cdot 10^{8}$ & 4.089 & 4.188 & 4.247 & 4.662 & 253 & 253 & $22 \cdot 10^{6}$ & 8 \\ \hline
103 & $2 \cdot 10^{5}$ & 4.108 & 4.206 & 4.264 & 4.681 & 271 & 37 & $22 \cdot 10^{6}$ & 8 \\ \hline
107 & $3 \cdot 10^{3}$ & 4.145 & 4.249 & 4.323 & 4.717 & 42 & 42 & $21 \cdot 10^{6}$ & 8 \\ \hline
109 & $8 \cdot 10^{3}$ & 4.163 & 4.270 & 4.302 & 4.735 & 147 & 85 & $20 \cdot 10^{6}$ & 7 \\ \hline
113 & $9 \cdot 10^{5}$ & 4.198 & 4.305 & 4.395 & 4.770 & 117 & 33 & $20 \cdot 10^{6}$ & 8 \\ \hline
127 & $4 \cdot 10^{3}$ & 4.311 & 4.428 & 4.507 & 4.883 & 47 & 47 & $18 \cdot 10^{6}$ & 6 \\ \hline
131 & $2 \cdot 10^{9}$ & 4.341 & 4.464 & 4.550 & 4.913 & 193 & 193 & $17 \cdot 10^{6}$ & 6 \\ \hline
137 & $5 \cdot 10^{5}$ & 4.385 & 4.508 & 4.645 & 4.957 & 62 & 62 & $16 \cdot 10^{6}$ & 6 \\ \hline
139 & $7 \cdot 10^{3}$ & 4.399 & 4.527 & 4.551 & 4.971 & 177 & 177 & $16 \cdot 10^{6}$ & 6 \\ \hline
149 & $224$ & 4.467 & 4.607 & 4.634 & 5.038 & 181 & 19 & $15 \cdot 10^{6}$ & 5 \\ \hline
151 & $122$ & 4.480 & 4.621 & 4.624 & 5.051 & 71 & 71 & $15 \cdot 10^{6}$ & 5 \\ \hline
157 & $821$ & 4.518 & 4.657 & 4.687 & 5.089 & 71 & 71 & $14 \cdot 10^{6}$ & 5 \\ \hline
163 & $10$ & 4.555 & 4.701 & 4.701 & 5.126 & 1 &  -- & $14 \cdot 10^{6}$ & 6 \\ \hline
167 & $3 \cdot 10^{4}$ & 4.578 & 4.733 & 4.766 & 5.149 & 163 & 163 & $13 \cdot 10^{6}$ & 5 \\ \hline
173 & $94$ & 4.613 & 4.768 & 4.768 & 5.184 & 253 &  -- & $13 \cdot 10^{6}$ & 5 \\ \hline
179 & $15$ & 4.646 & 4.806 & 4.806 & 5.217 & 263 &  -- & $13 \cdot 10^{6}$ & 5 \\ \hline
181 & $8 \cdot 10^{3}$ & 4.657 & 4.821 & 4.856 & 5.228 & 145 & 23 & $12 \cdot 10^{6}$ & 5 \\ \hline
191 & $705$ & 4.710 & 4.881 & 4.949 & 5.280 & 29 & 29 & $12 \cdot 10^{6}$ & 4 \\ \hline
193 & $22$ & 4.720 & 4.895 & 4.895 & 5.291 & 17 &  -- & $12 \cdot 10^{6}$ & 5 \\ \hline
197 & $490$ & 4.740 & 4.913 & 4.940 & 5.311 & 61 & 61 & $11 \cdot 10^{6}$ & 4 \\ \hline
199 & $59$ & 4.750 & 4.930 & 4.930 & 5.321 & 18 &  -- & $11 \cdot 10^{6}$ & 5 \\ \hline
211 & $18$ & 4.808 & 4.992 & 4.992 & 5.378 & 25 &  -- & $11 \cdot 10^{6}$ & 4 \\ \hline
223 & $205$ & 4.862 & 5.057 & 5.069 & 5.432 & 61 & 61 & $10 \cdot 10^{6}$ & 4 \\ \hline
227 & $11$ & 4.879 & 5.076 & 5.076 & 5.449 & 1 &  -- & $10 \cdot 10^{6}$ & 4 \\ \hline
229 & $28$ & 4.888 & 5.088 & 5.088 & 5.458 & 17 &  -- & $9.8 \cdot 10^{6}$ & 4 \\ \hline
233 & $14$ & 4.905 & 5.108 & 5.108 & 5.475 & 87 &  -- & $10 \cdot 10^{6}$ & 4 \\ \hline
239 & $53$ & 4.930 & 5.134 & 5.134 & 5.500 & 33 &  -- & $9.9 \cdot 10^{6}$ & 4 \\ \hline
241 & $2 \cdot 10^{3}$ & 4.938 & 5.142 & 5.222 & 5.508 & 31 & 31 & $9.9 \cdot 10^{6}$ & 3 \\ \hline
251 & $12$ & 4.978 & 5.188 & 5.188 & 5.548 & 1 &  -- & $9.9 \cdot 10^{6}$ & 4 \\ \hline
257 & $254$ & 5.002 & 5.217 & 5.230 & 5.571 & 73 & 73 & $9.9 \cdot 10^{6}$ & 4 \\ \hline
263 & $12$ & 5.024 & 5.235 & 5.235 & 5.594 & 1 &  -- & $10 \cdot 10^{6}$ & 4 \\ \hline
269 & $21$ & 5.046 & 5.264 & 5.264 & 5.616 & 7 &  -- & $9.9 \cdot 10^{6}$ & 4 \\ \hline
271 & $85$ & 5.054 & 5.273 & 5.273 & 5.623 & 13 &  -- & $10 \cdot 10^{6}$ & 4 \\ \hline
277 & $2 \cdot 10^{3}$ & 5.075 & 5.299 & 5.336 & 5.645 & 73 & 73 & $9.9 \cdot 10^{6}$ & 4 \\ \hline
281 & $42$ & 5.089 & 5.313 & 5.313 & 5.659 & 13 &  -- & $9.9 \cdot 10^{6}$ & 4 \\ \hline
283 & $225$ & 5.096 & 5.322 & 5.340 & 5.666 & 45 & 45 & $10 \cdot 10^{6}$ & 3 \\ \hline
293 & $12$ & 5.131 & 5.361 & 5.361 & 5.700 & 1 &  -- & $10 \cdot 10^{6}$ & 3 \\ \hline
307 & $20$ & 5.177 & 5.414 & 5.414 & 5.746 & 7 &  -- & $9.9 \cdot 10^{6}$ & 3 \\ \hline
311 & $12$ & 5.189 & 5.432 & 5.432 & 5.759 & 1 &  -- & $9.9 \cdot 10^{6}$ & 3 \\ \hline
313 & $2 \cdot 10^{5}$ & 5.196 & 5.434 & 5.489 & 5.765 & 129 & 129 & $10 \cdot 10^{6}$ & 4 \\ \hline
317 & $13$ & 5.208 & 5.454 & 5.454 & 5.778 & 1 &  -- & $10 \cdot 10^{6}$ & 3 \\ \hline
331 & $27$ & 5.251 & 5.504 & 5.504 & 5.820 & 9 &  -- & $9.9 \cdot 10^{6}$ & 3 \\ \hline
\end{tabular}
\caption{Data for $101 \leq n \leq 331$, prime}
\label{table:4}
\end{table}

\begin{table}[h!]
\centering
\begin{tabular}{|c|c|c|c|c|c|c|c|c|c|}  \hline
$n$ & $\cC_{1,n}$ & $\log \cD_{\text{Chud},n}$ & $\log \cD_{1,n}$ & $\log \cD_{2,n}$ & $\log n\mu_{n}$ & $r_{1,\text{max}}$ & $r_{2,\text{max}}$ & $r_{\text{comp}}$ & $N$ \\ \hline \hline
337 & $14$ & 5.269 & 5.522 & 5.522 & 5.838 & 5 &  -- & $10 \cdot 10^{6}$ & 3 \\ \hline
347 & $63$ & 5.298 & 5.555 & 5.555 & 5.867 & 25 &  -- & $9.9 \cdot 10^{6}$ & 3 \\ \hline
349 & $58$ & 5.303 & 5.564 & 5.564 & 5.872 & 25 &  -- & $10 \cdot 10^{6}$ & 3 \\ \hline
353 & $13$ & 5.315 & 5.581 & 5.581 & 5.884 & 1 &  -- & $9.9 \cdot 10^{6}$ & 3 \\ \hline
359 & $18$ & 5.331 & 5.599 & 5.599 & 5.900 & 7 &  -- & $10 \cdot 10^{6}$ & 3 \\ \hline
367 & $13$ & 5.353 & 5.617 & 5.617 & 5.922 & 1 &  -- & $9.9 \cdot 10^{6}$ & 3 \\ \hline
373 & $216$ & 5.369 & 5.640 & 5.654 & 5.938 & 59 & 59 & $9.9 \cdot 10^{6}$ & 3 \\ \hline
379 & $13$ & 5.385 & 5.654 & 5.654 & 5.954 & 1 &  -- & $9.9 \cdot 10^{6}$ & 3 \\ \hline
383 & $13$ & 5.395 & 5.669 & 5.669 & 5.964 & 1 &  -- & $10 \cdot 10^{6}$ & 3 \\ \hline
389 & $13$ & 5.410 & 5.691 & 5.691 & 5.979 & 1 &  -- & $10 \cdot 10^{6}$ & 3 \\ \hline
397 & $23$ & 5.431 & 5.716 & 5.716 & 6.000 & 9 &  -- & $9.8 \cdot 10^{6}$ & 3 \\ \hline
401 & $163$ & 5.441 & 5.723 & 5.731 & 6.009 & 65 & 65 & $10 \cdot 10^{6}$ & 3 \\ \hline
409 & $26$ & 5.460 & 5.746 & 5.746 & 6.029 & 35 &  -- & $10 \cdot 10^{6}$ & 3 \\ \hline
419 & $21$ & 5.484 & 5.775 & 5.775 & 6.053 & 28 &  -- & $10 \cdot 10^{6}$ & 3 \\ \hline
421 & $14$ & 5.489 & 5.775 & 5.775 & 6.058 & 1 &  -- & $10 \cdot 10^{6}$ & 3 \\ \hline
431 & $42$ & 5.512 & 5.810 & 5.810 & 6.081 & 13 &  -- & $9.9 \cdot 10^{6}$ & 3 \\ \hline
433 & $17$ & 5.516 & 5.811 & 5.811 & 6.085 & 7 &  -- & $10 \cdot 10^{6}$ & 3 \\ \hline
439 & $14$ & 5.530 & 5.829 & 5.829 & 6.099 & 1 &  -- & $9.9 \cdot 10^{6}$ & 3 \\ \hline
443 & $23$ & 5.539 & 5.839 & 5.839 & 6.108 & 31 &  -- & $9.9 \cdot 10^{6}$ & 3 \\ \hline
449 & $14$ & 5.552 & 5.854 & 5.854 & 6.121 & 1 &  -- & $9.9 \cdot 10^{6}$ & 3 \\ \hline
457 & $14$ & 5.570 & 5.882 & 5.882 & 6.139 & 1 &  -- & $9.9 \cdot 10^{6}$ & 3 \\ \hline
461 & $14$ & 5.578 & 5.889 & 5.889 & 6.147 & 1 &  -- & $9.9 \cdot 10^{6}$ & 3 \\ \hline
463 & $14$ & 5.583 & 5.897 & 5.897 & 6.152 & 1 &  -- & $9.9 \cdot 10^{6}$ & 3 \\ \hline
467 & $15$ & 5.591 & 5.904 & 5.904 & 6.160 & 7 &  -- & $9.9 \cdot 10^{6}$ & 3 \\ \hline
479 & $14$ & 5.616 & 5.934 & 5.934 & 6.185 & 1 &  -- & $10 \cdot 10^{6}$ & 3 \\ \hline
487 & $14$ & 5.633 & 5.956 & 5.956 & 6.201 & 1 &  -- & $10 \cdot 10^{6}$ & 2 \\ \hline
491 & $14$ & 5.641 & 5.971 & 5.971 & 6.210 & 1 &  -- & $10 \cdot 10^{6}$ & 3 \\ \hline
499 & $167$ & 5.657 & 5.986 & 6.002 & 6.226 & 33 & 33 & $9.9 \cdot 10^{6}$ & 2 \\ \hline
503 & $15$ & 5.665 & 5.992 & 5.992 & 6.233 & 1 &  -- & $10 \cdot 10^{6}$ & 2 \\ \hline
509 & $15$ & 5.677 & 6.014 & 6.014 & 6.245 & 1 &  -- & $9.9 \cdot 10^{6}$ & 3 \\ \hline
521 & $15$ & 5.700 & 6.051 & 6.051 & 6.268 & 1 &  -- & $10 \cdot 10^{6}$ & 2 \\ \hline
523 & $15$ & 5.703 & 6.046 & 6.046 & 6.272 & 1 &  -- & $10 \cdot 10^{6}$ & 2 \\ \hline
541 & $15$ & 5.737 & 6.082 & 6.082 & 6.306 & 1 &  -- & $10 \cdot 10^{6}$ & 2 \\ \hline
547 & $15$ & 5.748 & 6.095 & 6.095 & 6.316 & 1 &  -- & $10 \cdot 10^{6}$ & 3 \\ \hline
557 & $15$ & 5.766 & 6.117 & 6.117 & 6.334 & 1 &  -- & $10 \cdot 10^{6}$ & 2 \\ \hline
563 & $15$ & 5.777 & 6.132 & 6.132 & 6.345 & 1 &  -- & $10 \cdot 10^{6}$ & 2 \\ \hline
569 & $15$ & 5.787 & 6.144 & 6.144 & 6.356 & 1 &  -- & $9.9 \cdot 10^{6}$ & 2 \\ \hline
571 & $18$ & 5.791 & 6.150 & 6.150 & 6.359 & 11 &  -- & $10 \cdot 10^{6}$ & 2 \\ \hline
577 & $15$ & 5.801 & 6.168 & 6.168 & 6.369 & 1 &  -- & $9.9 \cdot 10^{6}$ & 2 \\ \hline
587 & $15$ & 5.818 & 6.185 & 6.185 & 6.386 & 1 &  -- & $9.9 \cdot 10^{6}$ & 2 \\ \hline
593 & $15$ & 5.828 & 6.204 & 6.204 & 6.396 & 1 &  -- & $10 \cdot 10^{6}$ & 3 \\ \hline
599 & $15$ & 5.838 & 6.210 & 6.210 & 6.406 & 1 &  -- & $10 \cdot 10^{6}$ & 2 \\ \hline
\end{tabular}
\caption{Data for $337 \leq n \leq 599$, prime}
\label{table:5}
\end{table}

\begin{table}[h!]
\centering
\begin{tabular}{|c|c|c|c|c|c|c|c|c|c|}  \hline
$n$ & $\cC_{1,n}$ & $\log \cD_{\text{Chud},n}$ & $\log \cD_{1,n}$ & $\log \cD_{2,n}$ & $\log n\mu_{n}$ & $r_{1,\text{max}}$ & $r_{2,\text{max}}$ & $r_{\text{comp}}$ & $N$ \\ \hline \hline
601 & $15$ & 5.841 & 6.226 & 6.226 & 6.410 & 1 &  -- & $9.9 \cdot 10^{6}$ & 2 \\ \hline
607 & $15$ & 5.851 & 6.234 & 6.234 & 6.420 & 1 &  -- & $10 \cdot 10^{6}$ & 2 \\ \hline
613 & $15$ & 5.861 & 6.242 & 6.242 & 6.429 & 1 &  -- & $10 \cdot 10^{6}$ & 3 \\ \hline
617 & $15$ & 5.867 & 6.244 & 6.244 & 6.436 & 1 &  -- & $9.9 \cdot 10^{6}$ & 2 \\ \hline
619 & $15$ & 5.871 & 6.244 & 6.244 & 6.439 & 1 &  -- & $9.9 \cdot 10^{6}$ & 2 \\ \hline
631 & $15$ & 5.890 & 6.280 & 6.280 & 6.458 & 1 &  -- & $10 \cdot 10^{6}$ & 2 \\ \hline
641 & $15$ & 5.905 & 6.296 & 6.296 & 6.474 & 1 &  -- & $10 \cdot 10^{6}$ & 2 \\ \hline
643 & $15$ & 5.908 & 6.308 & 6.308 & 6.477 & 1 &  -- & $10 \cdot 10^{6}$ & 2 \\ \hline
647 & $16$ & 5.914 & 6.302 & 6.302 & 6.483 & 1 &  -- & $10 \cdot 10^{6}$ & 2 \\ \hline
653 & $16$ & 5.924 & 6.312 & 6.312 & 6.492 & 1 &  -- & $10 \cdot 10^{6}$ & 2 \\ \hline
659 & $16$ & 5.933 & 6.329 & 6.329 & 6.501 & 1 &  -- & $10 \cdot 10^{6}$ & 2 \\ \hline
661 & $16$ & 5.936 & 6.326 & 6.326 & 6.504 & 1 &  -- & $10 \cdot 10^{6}$ & 2 \\ \hline
673 & $15$ & 5.954 & 6.378 & 6.378 & 6.522 & 1 &  -- & $10 \cdot 10^{6}$ & 2 \\ \hline
677 & $16$ & 5.959 & 6.365 & 6.365 & 6.528 & 1 &  -- & $10 \cdot 10^{6}$ & 2 \\ \hline
683 & $16$ & 5.968 & 6.372 & 6.372 & 6.537 & 1 &  -- & $10 \cdot 10^{6}$ & 2 \\ \hline
691 & $16$ & 5.980 & 6.395 & 6.395 & 6.548 & 1 &  -- & $9.9 \cdot 10^{6}$ & 2 \\ \hline
701 & $16$ & 5.994 & 6.413 & 6.413 & 6.562 & 1 &  -- & $10 \cdot 10^{6}$ & 2 \\ \hline
709 & $16$ & 6.005 & 6.417 & 6.417 & 6.574 & 1 &  -- & $9.9 \cdot 10^{6}$ & 2 \\ \hline
719 & $16$ & 6.019 & 6.447 & 6.447 & 6.588 & 1 &  -- & $10 \cdot 10^{6}$ & 2 \\ \hline
727 & $16$ & 6.030 & 6.458 & 6.458 & 6.599 & 1 &  -- & $9.9 \cdot 10^{6}$ & 2 \\ \hline
733 & $16$ & 6.038 & 6.462 & 6.462 & 6.607 & 1 &  -- & $9.9 \cdot 10^{6}$ & 2 \\ \hline
739 & $16$ & 6.046 & 6.474 & 6.474 & 6.615 & 1 &  -- & $9.9 \cdot 10^{6}$ & 2 \\ \hline
743 & $16$ & 6.052 & 6.489 & 6.489 & 6.620 & 1 &  -- & $10 \cdot 10^{6}$ & 2 \\ \hline
751 & $16$ & 6.062 & 6.499 & 6.499 & 6.631 & 1 &  -- & $10 \cdot 10^{6}$ & 2 \\ \hline
757 & $16$ & 6.070 & 6.501 & 6.501 & 6.639 & 1 &  -- & $9.9 \cdot 10^{6}$ & 2 \\ \hline
761 & $16$ & 6.076 & 6.523 & 6.523 & 6.644 & 1 &  -- & $10 \cdot 10^{6}$ & 2 \\ \hline
769 & $16$ & 6.086 & 6.528 & 6.528 & 6.654 & 1 &  -- & $9.9 \cdot 10^{6}$ & 2 \\ \hline
773 & $16$ & 6.091 & 6.540 & 6.540 & 6.659 & 1 &  -- & $10 \cdot 10^{6}$ & 2 \\ \hline
787 & $16$ & 6.109 & 6.564 & 6.564 & 6.677 & 1 &  -- & $10 \cdot 10^{6}$ & 2 \\ \hline
797 & $16$ & 6.122 & 6.580 & 6.580 & 6.690 & 1 &  -- & $9.9 \cdot 10^{6}$ & 2 \\ \hline
809 & $16$ & 6.136 & 6.599 & 6.599 & 6.705 & 1 &  -- & $10 \cdot 10^{6}$ & 2 \\ \hline
811 & $16$ & 6.139 & 6.603 & 6.603 & 6.707 & 1 &  -- & $10 \cdot 10^{6}$ & 2 \\ \hline
821 & $16$ & 6.151 & 6.615 & 6.615 & 6.719 & 1 &  -- & $10 \cdot 10^{6}$ & 2 \\ \hline
823 & $16$ & 6.153 & 6.621 & 6.621 & 6.722 & 1 &  -- & $10 \cdot 10^{6}$ & 2 \\ \hline
827 & $16$ & 6.158 & 6.627 & 6.627 & 6.726 & 1 &  -- & $10 \cdot 10^{6}$ & 2 \\ \hline
829 & $16$ & 6.161 & 6.641 & 6.641 & 6.729 & 1 &  -- & $10 \cdot 10^{6}$ & 2 \\ \hline
839 & $16$ & 6.173 & 6.653 & 6.653 & 6.741 & 1 &  -- & $10 \cdot 10^{6}$ & 2 \\ \hline
853 & $16$ & 6.189 & 6.683 & 6.683 & 6.757 & 1 &  -- & $9.9 \cdot 10^{6}$ & 2 \\ \hline
857 & $16$ & 6.194 & 6.678 & 6.678 & 6.762 & 1 &  -- & $10 \cdot 10^{6}$ & 2 \\ \hline
859 & $16$ & 6.196 & 6.677 & 6.677 & 6.764 & 1 &  -- & $10 \cdot 10^{6}$ & 2 \\ \hline
863 & $16$ & 6.201 & 6.681 & 6.681 & 6.769 & 1 &  -- & $9.9 \cdot 10^{6}$ & 2 \\ \hline
877 & $16$ & 6.217 & 6.706 & 6.706 & 6.785 & 1 &  -- & $9.9 \cdot 10^{6}$ & 2 \\ \hline
\end{tabular}
\caption{Data for $601 \leq n \leq 877$, prime}
\label{table:6}
\end{table}

\begin{table}[h!]
\centering
\begin{tabular}{|c|c|c|c|c|c|c|c|c|c|}  \hline
$n$ & $\cC_{1,n}$ & $\log \cD_{\text{Chud},n}$ & $\log \cD_{1,n}$ & $\log \cD_{2,n}$ & $\log n\mu_{n}$ & $r_{1,\text{max}}$ & $r_{2,\text{max}}$ & $r_{\text{comp}}$ & $N$ \\ \hline \hline
881 & $16$ & 6.221 & 6.710 & 6.710 & 6.789 & 1 &  -- & $10 \cdot 10^{6}$ & 2 \\ \hline
883 & $16$ & 6.223 & 6.723 & 6.723 & 6.792 & 1 &  -- & $9.9 \cdot 10^{6}$ & 2 \\ \hline
887 & $16$ & 6.228 & 6.723 & 6.723 & 6.796 & 1 &  -- & $10 \cdot 10^{6}$ & 2 \\ \hline
907 & $16$ & 6.250 & 6.751 & 6.751 & 6.818 & 1 &  -- & $10 \cdot 10^{6}$ & 2 \\ \hline
911 & $16$ & 6.254 & 6.761 & 6.761 & 6.823 & 1 &  -- & $10 \cdot 10^{6}$ & 2 \\ \hline
919 & $16$ & 6.263 & 6.774 & 6.774 & 6.831 & 1 &  -- & $10 \cdot 10^{6}$ & 2 \\ \hline
929 & $16$ & 6.274 & 6.789 & 6.789 & 6.842 & 1 &  -- & $10 \cdot 10^{6}$ & 2 \\ \hline
937 & $16$ & 6.282 & 6.806 & 6.806 & 6.850 & 1 &  -- & $10 \cdot 10^{6}$ & 2 \\ \hline
941 & $16$ & 6.287 & 6.816 & 6.816 & 6.855 & 1 &  -- & $10 \cdot 10^{6}$ & 2 \\ \hline
947 & $17$ & 6.293 & 6.813 & 6.813 & 6.861 & 1 &  -- & $10 \cdot 10^{6}$ & 2 \\ \hline
953 & $16$ & 6.299 & 6.827 & 6.827 & 6.867 & 1 &  -- & $10 \cdot 10^{6}$ & 2 \\ \hline
967 & $16$ & 6.314 & 6.848 & 6.848 & 6.882 & 1 &  -- & $9.9 \cdot 10^{6}$ & 2 \\ \hline
971 & $17$ & 6.318 & 6.847 & 6.847 & 6.886 & 1 &  -- & $10 \cdot 10^{6}$ & 2 \\ \hline
977 & $17$ & 6.324 & 6.857 & 6.857 & 6.892 & 1 &  -- & $10 \cdot 10^{6}$ & 2 \\ \hline
983 & $17$ & 6.330 & 6.867 & 6.867 & 6.898 & 1 &  -- & $10 \cdot 10^{6}$ & 2 \\ \hline
991 & $16$ & 6.338 & 6.889 & 6.889 & 6.906 & 1 &  -- & $10 \cdot 10^{6}$ & 1 \\ \hline
997 & $17$ & 6.344 & 6.884 & 6.884 & 6.912 & 1 &  -- & $10 \cdot 10^{6}$ & 2 \\ \hline
1009 & $17$ & 6.356 & 6.905 & 6.905 & 6.924 & 1 &  -- & $9.9 \cdot 10^{6}$ & 2 \\ \hline
\end{tabular}
\caption{Data for $881 \leq n \leq 1009$, prime}
\label{table:7}
\end{table}

\end{document}